\documentclass[onefignum,onetabnum]{siamart171218}

\usepackage{lipsum}
\usepackage{amsfonts}
\usepackage{graphicx}
\usepackage{epstopdf}
\usepackage{algorithm}
\usepackage{algcompatible}
\usepackage[noend]{algpseudocode}
\usepackage{cite}
\usepackage{amsmath,amssymb}
\usepackage{graphicx}
\usepackage{textcomp}
\usepackage{tikz}
\usepackage{pgfplots}
\usepackage{xcolor}
\usetikzlibrary{matrix,arrows,calc,positioning,shapes,decorations.pathreplacing}
\usepackage{graphicx}
\usepackage{amsopn}

\ifpdf
  \DeclareGraphicsExtensions{.eps,.pdf,.png,.jpg}
\else
  \DeclareGraphicsExtensions{.eps}
\fi

\newsiamremark{remark}{Remark}
\newsiamremark{hypothesis}{Hypothesis}
\crefname{hypothesis}{Hypothesis}{Hypotheses}
\newsiamthm{claim}{Claim}

\newsiamthm{example}{Example}

\headers{Reduction in Optimal Control Problems with broken symmetry}{E. Stratoglou, A. Anahory Simoes, L. Colombo}

\title{Reduction in optimal control with broken symmetry for collision and obstacle avoidance of multi-agent system on Lie groups\thanks{Submitted to the editors DATE.
\funding{The authors acknowledge financial support from the Spanish Ministry of Science
and Innovation, under grants PID2019-106715GB-C21, MTM2016-76702-P.}}}

\author{Efstratios Stratoglou\thanks{Programa de Doctorado de Autom\'atica y Rob\'otica, Universidad Polit\'ecnica de Madrid (UPM), 28006 Madrid, Spain
  (\email{ef.stratoglou@alumnos.upm.es}).}
\and Alexandre Anahory Simoes\thanks{Centre for Automation and Robotics (CSIC-UPM), Ctra. M300 Campo Real, Km 0,200, Arganda del Rey - 28500 Madrid, Spain  (\email{alexandre.anahory@car.upm-csic.es}, \email{leonardo.colombo@csic.es}).}
\and Leonardo J. Colombo\footnotemark[3]}

\def\BibTeX{{\rm B\kern-.05em{\sc i\kern-.025em b}\kern-.08em
    T\kern-.1667em\lower.7ex\hbox{E}\kern-.125emX}}
\DeclareMathOperator{\Cay}{Cay}
\DeclareMathOperator{\dCay}{dCay}
\newcommand{\g}{\mathfrak{g}}

\newcommand{\Vezero}{V_{i,0}^{\text{ext}}}
\newcommand{\R}{\mathbb{R}}
\newcommand{\Ad}{\text{Ad}}
\newcommand{\e}{\mbox{exp}}
\newcommand{\ad}{\text{ad}}
\newcommand{\se}{\mathfrak{se}(2)}
\DeclareMathOperator{\spn}{span}

\newcommand*{\addFileDependency}[1]{
  \typeout{(#1)}
  \@addtofilelist{#1}
  \IfFileExists{#1}{}{\typeout{No file #1.}}
}

\ifpdf
\hypersetup{
  pdftitle={Reduction in optimal control problems with symmetry breaking},
  pdfauthor={A. Anahory Simoes, E. Stratoglou, and L. J. Colombo}
}
\fi

\begin{document}

\maketitle

\begin{abstract}
We study the reduction by symmetry for optimality conditions in optimal control problems of left-invariant affine multi-agent control systems, with partial symmetry breaking cost functions for continuous-time and discrete-time systems. We recast the optimal control problem as a constrained variational problem with a partial symmetry breaking Lagrangian and obtain the reduced optimality conditions from a reduced variational principle via symmetry reduction techniques in both settings, continuous-time, and discrete-time. We apply the results to a collision and obstacle avoidance problem for multiple vehicles evolving on $SE(2)$ in the presence of a static obstacle.
\end{abstract}



\section{Introduction}
\label{sec:introduction}
Lie groups symmetries appear naturally in many control systems problems \cite{BlControlled, bonnabel1, contreras, eche, arcak, grizzle, LG1, LG2, manolo, LG3, bonnabel, saccon, arian}. Methods for trajectory tracking and estimation algorithms for the pose of mechanical systems evolving on Lie groups are commonly employed for improving the accuracy of simulations, as well as to avoid singularities by working with coordinate-free expressions in the associated Lie algebra of the Lie group which describes the motion of the systems as a set of ordinary differential equations depending on an arbitrary choice of the basis for the Lie algebra.

Optimization problems on Lie groups have a long history \cite{jurdjevic} and have been applied to many problems in control engineering. In practice, many robotic systems exhibit symmetries that can be exploited to reduce some of the complexities in the system
models, for instance, degrees of freedom. Symmetries in optimal control for systems on Lie groups have been studied in \cite{B, K, KM, leo1, Tomoki, Tomoki-CDC} among many others, mainly for applications in robotic and aerospace engineering, and in particular, for spacecraft attitude control and underwater vehicles \cite{Leonard1}.  While most of the applications of symmetry reduction provided in the literature focus on the single-agent situation, only a few works studied the relation between multi-agent systems and symmetry reduction (see for instance \cite{JK2, vasile1, vasile2, SCD}). In this work, we employ symmetry reduction to study optimal control problems with broken symmetry for multi-agent systems on Lie groups while agents avoid collisions and obstacles in the configuration space. 

 
 In our previous work, \cite{BCGO} we studied symmetry reduction in optimal control problems with broken symmetry for single-agent systems. In this work, we advance on the results of \cite{BCGO} by considering a multi-agent scenario. Hence, a new variational principle and reduction by symmetries performance are needed. The results in this paper are the Lagrangian/variational counterpart of those in \cite{SCD}; we also develop a discrete-time version of the results. From the Lagrangian point of view, we obtain the Euler--Poincar{\'e} equations from a constrained variational principle. By discretizing the variational principle in time, we obtain the discrete-time Lie--Poisson equations.


The main idea of the approached followed in this work to obtain the reduced optimality conditions is as follows. First note that the artificial potential $V^0$ used to prevent collisions with a fixed obstacle is not symmetry invariant. At the same time, we consider a representation of $G$ on a dual vector space $V^{*}$ at each node of the graph $\mathcal{G}$ which couples the neighbors of an agent in the unreduced Lagrangian for the optimal control problem, with a parameter depending on vectors in $V^{*}$ that are acted on by $G$. Hence, the neighbors are coupled with the vectors in $V^{*}$ that are acted by the adjoint representation. The associated action considered at this stage restores the full Lie group symmetry in the cost function from our optimal control problem, and hence we apply the semi-direct product reduction theory \cite{contreras, contreras2, GB, GBT, HMR, MRW1, MRW2}, to obtain the corresponding Euler-Poincar\'e system on the semi-direct product Lie algebra $\g\ltimes V^{*}$ at each node. This gives rise to a new system that finds no analogs in classical reduced-order models in optimal control of mechanical systems.


The paper is organized as follows. In Section \ref{sec2}, we introduce some preliminaries about  geometric mechanics on Lie groups and Lie group actions. In Section \ref{sec3}, we present the problem under study together with a motivating example. In Section \ref{sec4} we study the Euler--Poincar{\'e} reduction of optimal control problems for left-invariant multi-agent control systems on Lie groups with partial symmetry breaking cost functions. Furthermore, we consider the discrete-time framework and obtain the discrete-time Lie--Poisson equations in Section \ref{sec5}. In Section \ref{sec6} an example is considered to illustrate the theory. Finally, some concluding remarks are given in Section \ref{sec7}.

\section{Lie group actions and representations}\label{sec2}

Let $Q$ be the configuration space of a mechanical system, a differentiable manifold of dimension $d$ with local coordinates $q=(q^1,\ldots,q^d)$. Let $TQ$ be the tangent bundle of $Q$, locally described by positions and velocities for the system, $v_q=(q^1,\ldots, q^d,\dot{q}^{1},\ldots,\dot{q}^{d})\in TQ$ with $\hbox{dim}(TQ)=2d$. Let $T^{*}Q$ be its cotangent bundle, locally described by the positions and the momentum for the system, i.e., $(q,p)\in T^{*}Q$ with $\hbox{dim}(T^{*}Q)=2d$. The tangent bundle at $q\in Q$ has a vector space structure and it is denoted as $T_{q}Q$. The cotangent bundle at $q\in Q$ is just the dual space of $T_{q}Q$ and denoted as $T_{q}^{*}Q$.  In that sense, the momentum $p_q$ at $q\in Q$ can be thought as the dual of the velocity vector $v_q$ at the point $q\in Q$.The dynamics of a mechanical system is described by the equations of motion determined by a Lagrangian function $L:TQ\to\R$ given by $L(q,\dot{q})=K(q,\dot{q})-V(q)$, where $K:TQ\to\R$ denotes the kinetic energy and $V:Q\to\R$ the potential energy of the system. The equations of motion are given by the Euler-Lagrange equations 
$\displaystyle{\frac{d}{dt}\bigg(\frac{\partial L}{\partial\dot{q}^i}\bigg)=\frac{\partial L}{\partial q^i}, \; \; i=1,\dots,d}$, which determine a system of second-order differential equations. In the case the configuration space  of the system is a Lie group, Euler-Lagrange equations can be reduced to a first-order system of equations.

\begin{definition}Let $G$ be a Lie group and $Q$ a smooth manifold. A \textit{left-action} of $G$ on $Q$ is a smooth map $\Phi:G\times Q\to Q$ such that $\Phi(\bar{e},g)=g$ and $\Phi(h,\Phi(g,q))=\Phi(hg,q)$ for all $g,h\in G$ and $q\in Q$, where $\bar{e}$ is the identity of the group $G$ and the map $\Phi_g:Q\to Q$ given by $\Phi_g(q)=\Phi(g,q)$ is a diffeomorphism for all $g\in G$.\end{definition}

\begin{definition}A function $f:Q\to\R$ is called \textit{left invariant} under $\Phi_g$ if $f\circ\Phi_g=f$ for any $g\in G$.\end{definition}

For a finite dimensional Lie group $G$, its Lie algebra $\g$ is defined as the tangent space to $G$ at the identity, $\mathfrak{g}:=T_{\bar{e}}G$. Let $L_g:G\to G$ be the left translation of the element $g\in G$ given by $L_g(h)=gh$ where $h\in G$. $L_g$ is a diffeomorphism on $G$ and a left-action of $G$ on $G$ \cite{HSS}. Its tangent map  (i.e, the linearization or tangent lift) is denoted by $T_{h}L_{g}:T_{h}G\to T_{gh}G$. Similarly, the cotangent map (cotangent lift), is defined as $(T_{h}L_g)^{*}$, the dual map of the tangent lift denoted by $T_{h}^{*}L_{g}:T^{*}_{gh}G\to T^{*}_{h}G$, and determined by the relation $\langle(T_hL_g)^{*}(\alpha_{gh}), y_h\rangle=\langle\alpha_{gh},(T_hL_g)y_h\rangle$, $y_h\in T_{h}G$, $\alpha_{gh}\in T^{*}_{gh}G$. It is well known that the tangent and cotangent lift are Lie group actions. Here, $\langle\cdot,\cdot\rangle:V^{*}\times V\to\mathbb{R}$ with $V$ a finite dimensional vector space denotes the so-called \textit{natural pairing} between vectors and co-vectors and defined by $\langle y,x\rangle:=y\cdot x$ for $y\in V^{*}$ and $x\in V$, where $y$ is understood as a column vector and $x$ as a row vector. For a matrix Lie algebra $\langle y,x\rangle=y^{T}x$ (see \cite{HSS}, Section $2.3$). Using this pairing between vectors and co-vectors, for $g,h\in G$, $y\in\mathfrak{g}^{*}$ and $x\in\mathfrak{g}$, one can write $\langle T^{*}_{g}L_{g^{-1}}(y),T_{{\bar{e}}}L_{g}(x)\rangle=\langle y,x\rangle$.

Denote by $\ad^{*}:\mathfrak{g}\times\mathfrak{g}^{*}\to\mathfrak{g}^{*}$, $(\xi,\mu)\mapsto\ad^{*}_{\xi}\mu$ the \textit{co-adjoint operator}, defined by $\langle\ad_{\xi}^{*}\mu,\eta\rangle=\langle\mu,\ad_{\xi}\eta\rangle$ for all $\eta\in\mathfrak{g}$, where the $\ad:\mathfrak{g}\times\g\to\g$ denotes the adjoint operator on $\mathfrak{g}$ given by the Lie-bracket, i.e., $\ad_{\xi}\eta=[\xi,\eta]$, $\xi,\eta\in\g$. We also define the \textit{adjoint action} of $G$ on $\g$, denoted by  $\hbox{Ad}_{g}:\mathfrak{g}\to\mathfrak{g}$ and given by $\hbox{Ad}_{g}\chi:=g\chi g^{-1}$ where $\chi\in\mathfrak{g}$, and the \textit{co-adjoint action} of $G$ on $\g^{*}$, denoted by  $\hbox{Ad}_{g}^{*}:\mathfrak{g}^{*}\to\mathfrak{g}^{*}$, and given by $\langle\hbox{Ad}_{g}^{*}\alpha,\xi\rangle=\langle\alpha, \hbox{Ad}_{g}\xi\rangle$ with $\alpha\in\g^{*}$. 

For $q\in Q$, the \textit{isotropy} (or \textit{stabilizer} or \textit{symmetry}) group of $\Phi$ at $q$ is given by $G_q:=\{g\in G|\Phi_g(q)=q\}\subset G$. Since $\Phi_q(g)$ is a continuous,  $G_q=\Phi_q^{-1}(q)$ is a closed subgroup and hence a Lie subgroup of $G$ (see \cite{MR} Sec. $9.3$ for instance).

\begin{example}\label{example1}
Consider the special Euclidean group $\mathrm{SE}(2)$ of rotations and translations on the plane. Elements on $SE(2)$ can be described by transformations of $\mathbb{R}^2$ of the form $z \mapsto Rz+r$, with $r\in \mathbb{R}^2$ and $R\in SO(2)$. This transformation can be represented by $g=(R, r)$, for $\displaystyle{R=\left(
     \begin{array}{cc}
       \cos \theta & -\sin \theta\\
       \sin \theta & \cos \theta \\
     \end{array}
   \right)}$ and $r=[x,y]^{T}$. The composition law is $(R,r)\cdot(S,s)=(RS,Rs+r)$ with identity element $(I,0)$ and inverse $(R,r)^{-1}=(R^{-1},-R^{-1}r)$. Under this composition rule, $SE(2)$ has the structure of the semidirect product Lie group  $SO(2)\ltimes\mathbb{R}^2$. Here, as usual in the literature, we denote by $\ltimes$ the semidirect product of Lie groups.

The Lie algebra $\mathfrak{se}(2)$ of $SE(2)$ is determined by\newline
$\displaystyle{\mathfrak{se}(2)=\Big{\{}\left(
     \begin{array}{cc}
      A & b\\
       0 & 0 \\
     \end{array}
   \right)\Big{|} A\in \mathfrak{so}(2)\simeq\mathbb{R},\, b\in \mathbb{R}^2\Big{\}}}$.  In the following, for simplicity, we write $A=-a \mathbb{J} $, $a\in\mathbb{R}$, where $ \mathbb{J} =\left(
    \begin{array}{cc}
      0& 1\\
       -1 &0 \\
     \end{array}
   \right).$ Therefore,  we denote $\xi=(a,b)\in\mathfrak{se}(2)$. The adjoint action of $SE(2)$ on $\mathfrak{se}(2)$ is given by $\hbox{Ad}_{(R,r)}(a,b)=(a,a\mathbb{J}r+Rb)$ (see \cite{H}, pp. 153 for instance), so, $\hbox{Ad}_{(R,r)^{-1}}(a,b)=(a,R^{T}(b-a\mathbb{J}r))$.
\end{example}

Next, we provide the infinitesimal description of a Lie group action, which will be an important concept in the remainder of the paper. 

\begin{definition}
Given a Lie group action $\Phi:G\times Q\to Q$, for $\xi\in\mathfrak{g}$, the map $\Phi_{\hbox{exp}(t\xi)}:Q\to Q$ is a flow on $Q$. The corresponding vector field on $Q$, given by $\xi_Q(q):=\frac{d}{dt}\mid_{t=0}\Phi_{\hbox{exp} (t\xi)}(q)$ is called the \textit{infinitesimal generator} of the action corresponding to $\xi$.
\end{definition}

\begin{definition}\label{defLI}Denote by $\mathfrak{X}(G)$ the set of vector fields on $G$. A vector field $X\in\mathfrak{X}(G)$ is called \textit{left invariant} if $T_{h}L_g(X(h))=X(L_g(h))=X(gh)$ $\forall\,g,h\in G$.\end{definition}

In particular for $h={\bar{e}}$, this means that a vector field $X$ is left-invariant if $\dot{g}=X(g)=T_{{\bar{e}}}L_{g}\xi$ for $\xi=X({\bar{e}})\in\mathfrak{g}$. As $X$ is left invariant, $\xi=X({\bar{e}})=T_{g}L_{g^{-1}}\dot{g}$. The tangent map $T_{{\bar{e}}}L_g$ shifts vectors based at ${\bar{e}}$ to vectors based at $g\in G$. By doing this operation for every $g\in G$ we define a vector field as $\displaystyle{\xi}(g):=T_{{\bar{e}}}L_g(\xi)$ for $\xi:=X({\bar{e}})\in T_{{\bar{e}}}G$. Note that the vector field $\xi(g)$ is left invariant, because $\xi(hg)=T_{{\bar{e}}}(L_h\circ
L_g)\xi=(T_{g}L_h)\circ(T_{{\bar{e}}}L_g)\xi=T_gL_h \xi(g).$ 

\begin{example}
Consider the Euclidean Lie group $\mathbb{R}^{n}$ with the sum as group operation. For all $g\in\mathbb{R}^{n}$, the left translation $L_g$ is the usual translation on $\mathbb{R}^n$, that is, $L_g(h)=g+h$, $h\in\mathbb{R}^n$. So that, the tangent map to $L_g$ is the identity map on $\mathbb{R}^n$, that is, $T_{0}L_g=\hbox{id}_{T_{0}\mathbb{R}^n}$, where we are using that $T_{h}\mathbb{R}^n\simeq\mathbb{R}^n$ for all $h\in\mathbb{R}^n$, since $\mathbb{R}^n$ is a vector space. Therefore, left-invariant vector fields are constant vector fields, that is, $X=v_1\frac{\partial}{\partial x_1}+\ldots+v_n\frac{\partial}{\partial x_n}$ for $v=(v_1,\ldots, v_n)\in T_{0}\mathbb{R}^n$ and $x=(x_1,\ldots,x_n)\in\mathbb{R}^n$.
\end{example}


Consider a Lie group $G$, a vector space $V$ and the \textit{representation of $G$ on $V$} given by $\rho:G\times V\to V$, $(g,v)\mapsto\rho_g(v)$, which is a left action, and it is defined by the relation $\rho_{g_1}(\rho_{g_2}(v))=\rho_{g_1g_2}(v)$, $g_1,g_2\in G$. Its dual is given by $\rho^*:G\times V^*\to V^*,$ $(g,\alpha)\mapsto\rho^*_g(\alpha),$ satisfying $\langle\rho^*_{g^{-1}}(\alpha),v\rangle=\langle\alpha,\rho_{g^{-1}}(v)\rangle$.\


The \textit{infinitesimal generator} of the left action of $G$ on $V$ is $\rho':\g\times V\to V,$ $(\xi,v)\mapsto\rho'(\xi,v)=\frac{d}{dt}|_{t=0}\rho_{e^{t\xi}}(v).$ For every $v\in V$ consider the linear transformation $\rho'_v:\g\to V,$ $\xi\mapsto\rho'_v(\xi)=\rho'(\xi,v)$ and its dual $\rho'^*_v: V^*\to\g^*,$ $\alpha\mapsto\rho'^*_v(\alpha).$ The last transformation defines the \textit{momentum map} $\textbf{J}_V:V\times V^*\to\g^*,$ $(v,\alpha)\mapsto \textbf{J}_V(v,\alpha):=\rho'^*_v(\alpha)$ such that for every $\xi\in\g$, $\displaystyle{\langle \textbf{J}_V(v,\alpha),\xi\rangle=\langle\rho'^*_v(\alpha),\xi\rangle=\langle \alpha,\rho'_v(\xi)\rangle=\langle\alpha,\rho'(\xi,v)\rangle}$.

 For $\xi\in\g$, consider the map $\rho'_\xi:V\to V$, $v\mapsto\rho'_\xi(v)=\rho'(\xi,v)$ and its dual $\rho'^*_\xi:V^*\to V^*$, $\alpha\mapsto\rho'^*_\xi(\alpha)$ such that $\langle\rho'^*_\xi(\alpha),v\rangle=\langle\alpha,\rho'_\xi(v)\rangle$. Then this satisfies $\displaystyle{\langle\mathbf{J}_V(v,\alpha),\xi\rangle=\langle\alpha,\rho'(\xi,v)\rangle=\langle\alpha,\rho'_\xi(v)\rangle=\langle\rho'^*_\xi(\alpha),v\rangle}$. See \cite{HSS} and \cite{MR} for more details on the momentum map.
\begin{example}
Let $V=\mathfrak{g}$ and $\rho$ the adjoint representation of $G$ on $V$, i.e., $\rho_g=\hbox{Ad}_g$, for anly $g\in G$. So, for $\alpha\in V^{*}=\mathfrak{g}$, $\rho^{*}$ is the coadjoint representation of $G$ on $V^{*}$, i.e., $\rho_{g}^{*}=\hbox{Ad}_{g}^{*}$. We also have that the infinitesimal generator for the adjoint representation is $\rho_{\xi}'=\hbox{ad}_{\xi}$ for any $\xi\in V$ (see \cite{HSS}, Def. $6.4$, pp. $225$), and it follows that $\langle\mathbf{J}_{V}(x,\alpha),\xi\rangle = \langle\alpha,\ad_{\xi}x\rangle= \langle\alpha,-\ad_{x}\xi\rangle= \langle-\ad_{x}^{*}\alpha,\xi\rangle$, which gives $\mathbf{J}_{V}(x,\alpha) = -\ad_{x}^{*}\alpha$.

Similarly, if now $V=\mathfrak{g}^{*}$ and $\rho$ is the coadjoint representation of $G$ on $V$, i.e., $\rho_g=\hbox{Ad}^{*}_{g^{-1}}$, for any $g\in G$, then  $\rho_{\xi}'=-\ad_{\xi}^{*}$ for any $\xi\in\mathfrak{g}$. So, it follows that $\langle\mathbf{J}_{V}(x,\alpha),\xi\rangle = \langle-\ad_{\xi}\alpha,x\rangle= \langle x,\ad_{\alpha}\xi\rangle= \langle\ad_{\alpha}^{*}x,\xi\rangle$, which gives $\mathbf{J}_{V}(x,\alpha) = \ad_{\alpha}^{*}x$, for $\alpha\in\mathfrak{g}$ and $x\in\mathfrak{g}^{*}$.
\end{example}
\section{Problem Formulation}

\subsection{Left-invariant multi-agent control systems}\label{sec3}

Denote by $\mathcal{N}$ a set consisting of $s\geq 2$ free agents, and by $\mathcal{E}=\mathcal{N}\times\mathcal{N}$ the set describing the interaction between them. The neighboring relationships are described by an undirected graph $\mathcal{G}=(\mathcal{N},\mathcal{E})$, where $\mathcal{N}$ is the set of vertices and $\mathcal{E}$ the set of edges for $\mathcal{G}$. We further assume $\mathcal{G}$ is static and connected.  For every agent $i\in \mathcal{N}$ the set $\mathcal{N}_i=\{j\in\mathcal{N}:(i,j)\in\mathcal{E}\}$ denotes the neighbors of that agent. The agent $i\in\mathcal{N}$ evolves on an $n$-dimensional Lie group $G$ and its configuration is denoted by $g_i\in G$. We denote by $G^s$ and by $T_{\overline{e}}G^s=:\g^s$ the cartesian products of $s$ copies of $G$ and $\g$, respectively, where $\overline{e}=(\overline{e}_1,\overline{e}_2,\dots,\overline{e}_s)$ is the identity of $G^s$ with $\overline{e}_i$ being the identity element of the $i^{th}$-Lie group of $G^s$. The $i^{th}$-
Lie group as well as the $i^{th}$-Lie algebra will be denoted by $G_i$ and $\g_i$, respectively.

For each agent $i\in\mathcal{N}$ there is an associated left-invariant control system described by the kinematic equations \begin{equation}\label{kin-each-agent}
    \dot{g}_i=T_{\overline{e}_i}L_{g_i}(u_i), \;\; g_i(0)=g^0_i,
\end{equation} where $g_i(t)\in C^1([0,T],G_i), \; T\in\mathbb{R}$ fixed, $u_i\in\g_i$ is the control input and $g^i_0\in G$ is considered as the initial state condition. Note that while for each $i\in\mathcal{N}$, $\dim\g_i=n$ with $\g_i=\hbox{span}\{e_1^i,e_2^i,\ldots, e_n^i\}$, then the control inputs may be described by $u_i=[u_i^1,u_i^2,\dots,u_i^m]^T$, where $u_i(t)\in C^1([0,T],\g_i)$, with $m\leq n$. Hence, the control input for each agent is given by $\displaystyle{u_i(t)=e_0^i+\sum_{k=1}^{m}u^k_i(t)e_k^i}$, where $e_0^i\in\g_i.$ Thus, the left-invariant control systems \eqref{kin-each-agent} for each agent $i\in\mathcal{N}$ can be written as \begin{equation}\label{kin-each-agent-basis}
    \dot{g}_i(t)=g_i(t)\bigg(e_0+\sum_{k=1}^{m}u^k_i(t)e_k^i\bigg),\;\; g_i(0)=g^0_i.
\end{equation}

Note that the class of control systems described by \eqref{kin-each-agent-basis} capture underactuated as well as holonomic and nonholonomic constrained agents.

\subsection{Motivating Example}

   Consider the agents $i\in\mathcal{N}$ and $j\in\mathcal{N}_i$ represented as $g_k=(R_k,r_k)$, $k\in\{i,j\}$. Note that $g^{-1}_ig_j=(R_i^{T}R_j,-R_i^{T}(r_i-r_j))$, then, $\hbox{Ad}_{g_i^{-1}g_j}(1,0)=(1,\mathbb{J}R_i^{T}(r_j-r_i))$. The inner product on $\mathfrak{se}(2)$ is given by $\langle\xi_1,\xi_2\rangle=\hbox{tr}(\xi_1^T\xi_2)$ for $\xi_1,\xi_2\in\mathfrak{se}(2)$ and hence, the norm is given by $||\xi||=\sqrt{\hbox{tr}(\xi^T\xi)},$ for any $\xi\in\mathfrak{se}(2)$. For $\xi=(a,b)\in\mathfrak{se}(2)$ we can write the norm of $\xi$ as $||(a,b)||=\sqrt{2a^2+b^{T}b}$. Therefore, $||\hbox{Ad}_{g_i^{-1}g_j}(1,0)||^{2}=2+|\mathbb{J}R_i^{T}(r_j-r_i)|^{2}=2+|r_j-r_i|^2$, where we have used that $R,\mathbb{J}\in SO(2)$ for the last equality. Hence, it follows that  $|r_i-r_j|=\sqrt{||\hbox{Ad}_{g_i^{-1}g_j}(1,0)||^{2}-2}$.
   
   The previous computation shows that, if the interaction between agents is determined by a function depending on the distances between them, that is, $V_{ij}:G\times G\to\mathbb{R}$, is such that $V_{ij}(g_i,g_j)=V(|r_i-r_j|)$ for some $V:\mathbb{R}_{\geq 0}\to\mathbb{R}$; then, $V_{ij}$ is $SE(2)$-invariant, that is $V_{ij}(hg_i,hg_j)=V_{ij}(g_i,g_j)$. An alternative reasoning of this invariance property has been shown in \cite{O-SM}.
   
   Next, suppose that we wish to write the distance from an arbitrary point $r\in\mathbb{R}^2$ to a fixed point $x_0\in\mathbb{R}^2$ in terms of the adjoint action. Consider $\xi_0=(1,\mathbb{J}x_0)\in\mathfrak{se}(2)$. Then, for any $(R,r)\in SE(2)$, we have $\hbox{Ad}_{(R,r)^{-1}}(1,\mathbb{J}x_0)=(1,R^{T}\mathbb{J}(x_0-r))$, and therefore, $||\hbox{Ad}_{(R,r)^{-1}}(1,\mathbb{J}x_0)||^2=2+|x_0-r|^2$. Next, assume we have an obstacle avoidance function $V_i^{0}:SE(2)\to\mathbb{R}$ for each agent $i\in\mathcal{N}$ which can be written as $V_i^{0}(R_i,r_i)=V(|r_i|)$ with $V:\mathbb{R}_{\geq 0}\to\mathbb{R}$. Note that under this assumption, $V_i^{0}(R_i,0)$ may be chosen arbitrarily. Then, $V_i^0$ is not $SE(2)$-invariant, but it is $SO(2)$ invariant, i.e., $V_i^{0}(R_{0}R_i,r_i)=V_i^{0}(R_i,r_i)$ for any $R_0\in SO(2)$. Note also that $SO(2)$ is the isotropy group for the coadjoint action, that is, $SO(2)\simeq\{g\in SE(2)\mid Ad_{g}(\xi_{0}) = \xi_{0}\}$, therefore, the obstacle avoidance potential functions $V_i^{0}$ are invariant under the left action of the isotropy group.
   
   In this situation, one can redefine the potential function $V_i^{0}$ to make it $SE(2)$-invariant as follows. Consider $x_0=0$, Then, $\xi_0=(R_0,x_0)=(-\mathbb{J},0)\in\mathfrak{se}(2)$. Then $\hbox{Ad}_{(R_i,r_i)^{-1}}\xi_0=2+|r_i|^2$. Hence, $\displaystyle{V_i^{0}(R_i,r_i)=V(|r_i|)=V\left(\sqrt{||\hbox{Ad}_{(R_i,r_i)^{-1}}\xi_0||^2-2}\right)}$. This gives a motivation to define an extended obstacle avoidance function $V_{i,0}^{\textnormal{ext}}:G\ltimes V^{*}\to\mathbb{R}$ with $G\ltimes V^{*}=SE(2)\ltimes\mathfrak{se}(2)$ as  $\displaystyle{V_{i,0}^{\textnormal{ext}}(g_i,\xi):=V\left(\sqrt{||\hbox{Ad}_{g_i^{-1}}\xi||^2-2}\right)}$.
   
   Note that the extended obstacle avoidance function possesses now an $SE(2)$-symmetry (i.e., $V_{i,0}^{\textnormal{ext}}$ is invariant under a left action of $SE(2)$) since $V_{i,0}^{\textnormal{ext}}\circ\tilde{\Phi}=V_{i,0}^{\textnormal{ext}}(L_g h,\hbox{Ad}_{h}\xi)=V_{i,0}^{\textnormal{ext}}(R_i,\xi_i)$, for any $h\in SE(2)$, with left action $\tilde{\Phi}$ given by \begin{align}
		\tilde{\Phi}: SE(2) \times (SE(2) \times \mathfrak{se}(2))&\rightarrow SE(2)\times \mathfrak{se}(2),\\
		(h,(g,\xi))&\mapsto(L_{g}(h),\hbox{Ad}_h(\xi)).\nonumber
	\end{align}

	\subsection{Problem Setting}
The problem under study consists on finding reduced necessary conditions for optimality in an optimal control problem for \eqref{kin-each-agent} (or equivalently \eqref{kin-each-agent-basis}). These solution curves should minimize a cost function and prevent collisions among agents while they should also avoid static obstacles in the workspace. 


\textbf{Problem (collision and obstacle avoidance):} Find reduced optimality conditions on  $g(t)=(g_1(t),\dots,g_s(t))\in G^s$ and the controls $u(t)=(u_1(t),\dots,u_s(t))\in \g^s$ avoiding collision among the agents and obstacles (which will be defined shortly) in the workspace, and such that  $(g(t),u(t))\in G^s\times\g^s$ minimize the following cost function

\begin{equation}
\min_{(g,u)}\sum_{i=1}^s\int_0^T\Big(C_i(g_i(t),u_i(t))+V_i^{0}(g_i)+\frac{1}{2}\sum_{j\in \mathcal{N}_i}V_{ij}(g_i(t),g_j(t))\Big)dt \label{OCP}
\end{equation}
\sloppy subject to the kinematics  $\dot{g}_i(t)=T_{\overline{e}_i}L_{g_i(t)}(u_i(t))$, boundary conditions $g(0)=(g_1(0),\dots,g_n(0))=(g_1^0,\dots,g_s^0)$ and $g(T)=(g_1(T),\dots,g_s(T))=(g_1^T,\dots,g_s^T)$ with $T\in\mathbb{R}^{+}$ the final time, and under the following assumptions:

\begin{enumerate}
\item[(i)]{There is a left representation $\rho$ of $G$ on a vector space $V$.}
\item[(ii)]{$C_i: G\times\g\to\mathbb{R}$ are $G$-invariant functions for each  $i\in\mathcal{N}$ (under a suitable left action of $G$ on $G\times\g$, which will be defined shortly) and are also differentiable almost everywhere}.\label{a1}
\item[(iii)]{$V_{ij}: G\times G\to\mathbb{R}$ (collision avoidance potential functions) satisfying $V_{ij}=V_{ji}$ are $G$-invariant functions under $\Phi$, defined  by \begin{align}
		\Phi: G \times (G \times G)&\longrightarrow G\times G,\label{eq_phi}\\
		(g,(g_1, g_2))&\longmapsto(L_{g}(g_1),L_{g}(g_2)),\nonumber
	\end{align} i.e., $V_{ij}\circ\Phi_{g} = V_{ij}$, for any $g \in G$, that is,  $V_{ij}(L_{g}(g_i),L_g(g_j)) = V_{ij}(g_i,g_j)$, for any $(g_i,g_j)\in \mathcal{E}$, $j\in\mathcal{N}_i$ and they are also differentiable almost everywhere.}

\item[(iv)]{$V_{i}^{0}: G\to\mathbb{R}$ (obstacle avoidance potential functions) are not $G$-invariant functions and they are also differentiables almost everywhere, for $i\in\mathcal{N}$.}

\item[(v)]{The obstacle avoidance functions $V_{i}^{0}$ depend on a parameter $\alpha_{0}\in V^{*}$. Hence, we can define the extended potential function as $V_{i,0}^{\textnormal{ext}}: G\times V^{*}\to\mathbb{R}$, with $V_{i,0}^{\textnormal{ext}}(\cdot,\alpha_{0}) = V_{i}^{0}$, by making the parameter evolve - due to the group action - with initial value $\alpha_0$}.

\item[(vi)]{The extended obstacle avoidance functions are $G$-invariant under 
	$\tilde{\Phi}$, defined by \begin{align}
		\tilde{\Phi}: G \times (G \times V^{*})&\longrightarrow G\times V^{*},\label{eq_phi_tilde}\\
		(g,(h,\alpha))&\longmapsto(L_{g}(h),\rho_{g^{-1}}^{*}(\alpha)),\nonumber
	\end{align}  where $\rho_{g^{-1}}^{*}\in V^{*}$ is the adjoint of $\rho_{g^{-1}}\in V$, i.e., $V_{i,0}^{\textnormal{ext}}\circ\tilde{\Phi}_{g} = V_{i,0}^{\textnormal{ext}}$, for any $g \in G$, or $V_{i,0}^{\textnormal{ext}}(L_{g}(h),\rho_{g^{-1}}^{*}(\alpha)) = V_{i,0}^{\textnormal{ext}}(g,\alpha)$ where $\alpha\in V^{*}$}.

\item[(vii)]{The obstacle avoidance potential functions are invariant under the left action of the isotropy group
	\begin{equation}
		G_{\alpha_{0}}=\{g\in G\mid\rho_{g}^{*}(\alpha_{0}) = \alpha_{0}\}.
	\end{equation}\label{a6}}

\end{enumerate}

\vspace{-0.2cm}

 Note that $G\times\g$ is a trivial vector bundle over $G$ and define the left action of $G$ on $G\times\g$ as follows
\begin{align}
\Psi: G\times (G\times\g)&\longrightarrow (G\times\g),\nonumber\\
(g,(h,u))&\longmapsto(L_{g}(h),u).\label{eq_psi}
\end{align}

We further assume that each $C_i: G\times\g\to\mathbb{R}$ is $G$-invariant under \eqref{eq_psi}, i.e., $C_i\circ\Psi_{g} = C_i$, for any $g\in G$. In addition, each agent $i\in\mathcal{N}$ occupies a disk of radius $\overline{r}$ on $G$. This radius is chosen to be small enough so that all agents can be packed on $G$ and hence the potential functions are well defined and feasible for $\mathbf{d}(g_i(t),g_j(t))>2\overline{r}$ for all $t$, where $\mathbf{d}(\cdot,\cdot):G\times G\to\mathbb{R}$ denotes a distance function on $G$.

\section{Euler-Poincar\'e reduction for optimal control with broken symmetries}\label{sec4}
We next study reduced optimality conditions for extrema for the OCP. We address the problem as a constrained variational problem and obtain the Euler-Poincar\'{e} equations that normal extremal must satisfy in Theorem \ref{main-theo} and Proposition \ref{split prop}.

The optimal control problem (\ref{OCP}) can be solved as a constrained variational problem by introducing the Lagrangian multipliers $\lambda_{g_i}=T^*_{g_i}L_{g_i^{-1}}(\lambda_i(t))\in T^*_{g
_i}G$ where $\lambda_i\in C^1([0,T],\g^*)$ into the cost functional. 

Consider the dual of the Lie algebra, $\g^*=\hbox{span}\{e^1,\dots,e^m,e^{m+1},\dots,e^n\}$, with basis the dual basis of $\g$, and then $\displaystyle{\lambda_i=\sum_{k=m+1}^n\lambda^i_k e^k}$, where $\lambda_k^i$ are the components of the vector $\lambda_i$ in the given basis of the Lie algebra $\mathfrak{g}$. Thus, we define the Lagrangian $L:G^s\times\g^s\times(T^*_{g_i}G)^s\to\mathbb{R}$ by \begin{equation}
        L(g,u,\lambda)=\sum_{i=1}^s \bigg[C_i(g_i(t),u_i(t))+\langle\lambda_{g_i},T_{\overline{e}_i}L_{g_i}u_i\rangle\\ +V^0_i(g_i(t))+V_{i}(g)\bigg],
\end{equation} where $V_{i}:G^{s}\rightarrow \mathbb{R}$, 
\begin{equation*}
    V_{i}(g) = \frac{1}{2}\sum_{j\in \mathcal{N}_i}V_{ij}(g_i(t),g_j(t))
\end{equation*}

By assumption (v), the obstacle avoidance potential functions $V^0_i:G\to\mathbb{R}$ depends on a parameter $\alpha_0\in V^{*}$, so we can extend it to $V_{i,0}^{\text{ext}}:G\times V^{*}\to\mathbb{R}$ by making the parameter evolve under the Lie group action with $\alpha(0)=\alpha_0$, and therefore we can consider the extended Lagrangian function on $G^s\times\g^s\times(T^*_{g_i}G)^s\times V^*$,
\begin{equation*}
L_{\hbox{ext}}(g,u,\lambda,\alpha)=\sum_{i=1}^s \bigg[C_i(g_i(t),u_i(t)) + \langle\lambda_{g_i},T_{\overline{e}_i}L_{g_i}u_i\rangle
+\Vezero(g_i,\alpha_i)+V_{i}(g)\bigg]
\end{equation*}where $L_{\hbox{ext}}(g,u,\lambda,\alpha_0)=L(g,u,\lambda)$.

By assumptions (i) to (iv), and by taking advantage of the $G$-invariance of $C_i$, $V_{i,0}^{\text{ext}}$ and $V_{ij}$ (and so $V_i$)  we can define the reduced extended Lagrangian $\ell:G^{s-1}\times\g^s\times(\g^*)^s\times V^*\to\mathbb{R}$ by \begin{equation}\label{red-ext-Langrangian}
\begin{split}
\ell(g,u,\lambda,\alpha)=&\sum_{i=1}^s \bigg[C_i(u_i)+\langle\lambda_i,u_i\rangle+\Vezero(e_i,\alpha_i)+V_{i}(g)\bigg],\nonumber
\end{split}
\end{equation} \noindent defining $g_1$ to be the identity $\bar{e}_1$ on $G$, and where  $\ell(g,u,\lambda,\alpha)=L_{\text{ext}}(L_{g_i^{-1}}g,u,T^{*}_{g_i}L_{g_i^{-1}}\lambda_{g_i},\alpha)$ and $\alpha_i=\rho_{g_i}^*(\alpha)$. Note here the slight abuse of notation regarding the positions $g$. In the definition of the reduced Lagrangian $g\in G^{s-1}$ while in that of the Lagrangian $g\in G^s$. In this way, $\ell(g_2,\ldots,g_s,u,\lambda,\alpha)=L_{\text{ext}}(e,g_2,\ldots,g_s,u,\lambda, \alpha)$.

\begin{theorem}\label{main-theo}
For $s\geq 2$, a normal extrema for the OCP (\ref{OCP}) satisfies the following Euler-Poincar\'{e} equations

\begin{equation}\label{E-P eqs}
  		  \frac{d}{dt}\Big(\frac{\partial C_i}{\partial u_i}+\lambda_i\Big) =\ad^*_{u_i}\Big(\frac{\partial C_i}{\partial u_i}+\lambda_i\Big) +\textbf{J}_V\Big(\frac{\partial\Vezero}{\partial\alpha_i},\alpha_i\Big)+\Theta^{i}_{1}\sum_{k=1}^{s}T^*_{\overline{e}_i}L_{g_i}\bigg(\frac{\partial V_{k}}{\partial g_i}\bigg),
\end{equation}

\begin{equation}\label{E-P eqs param}
    \dot{\alpha_i}=\rho'^*_{u_i}(\alpha_i), \;\quad \alpha_i(0)=\rho^*_{g_i^0}(\alpha_0),
\end{equation}where $\textbf{J}_V:T^{*}V\simeq V\times V^{*}\to\g^{*}$ is the momentum map corresponding to the left action of $G$ on $V$ defined using the left representation $\rho$ of $G$ on $V$, and where $\Theta^{i}_{1}=0$ if $i=1$, otherwise $\Theta^{i}_{1}=1$.
\end{theorem}

\begin{proof}

Consider the variational principle 

\[\delta\int_0^TL_{ext}(g(t),u(t),\lambda(t),\alpha)dt=0,\]

\noindent which holds for variations of $g$, that vanishing at the endpoints, and $u$. Also, consider the constrained variational principle 

\begin{equation}\label{red-var-pr=0}
    \delta\int_0^T\ell(g(t),u(t),\lambda(t),\alpha(t))dt=0,
\end{equation} that holds for variations of $u$ and $\alpha$ with $\delta u=\dot{\eta}+\ad_u\eta$ and $\delta \alpha=\rho'^*_\eta(\alpha)$, where $\eta$ is a path of $\g^s$ that vanishes at the endpoints, i.e. $\eta(0)=\eta
(T)=0.$

The two variational principles are equivalent since the cost functions $C_i$ and the extended potential functions $\Vezero$ are $G$-invariant, i.e. $C_i\circ\Psi_g, \;\; \Vezero\circ\tilde{\Phi}=\Vezero$ and $\langle\lambda_{g_i},T_{\bar{e}_i}L_{g_i}u_i\rangle=\langle T^{*}_{g_i}L_{g_i^{-1}}\lambda_{g_i},u_i\rangle=\langle\lambda_i,u_i\rangle$. The variations $\delta g_i$ of $g_i$ induce and are induced by variations $\delta u_i=\dot{\eta}_i+\ad_{u_i}\eta_i$ with $\eta_i(0)=\eta_i(T)=0$, where $\eta_i=T_{g_i}L_{g_i}^{-1}(\delta g_i)$ and $\delta u_i=T_{g_i}L_{g_i^{-1}}(\dot{g}_i)$. Variations of $\alpha_i$ are given by $\delta\alpha_i=\rho'^*_{\eta_i}(\alpha_i)$.
So we have

\begin{equation} \label{red-var-pr}
    \begin{split}
        \delta\int_0^T\ell(g,u,\lambda,\alpha)dt=& \sum_{i=1}^s \int_0^T \Big\langle\frac{\partial C_i}{\partial u_i},\delta u_i\Big\rangle+ \langle\lambda_i,\delta u_i\rangle \\
        &+\Big\langle\frac{\partial\Vezero}{\partial\alpha_i},\delta\alpha_i\Big\rangle+ \sum_{k=2}^{s}\Big\langle\frac{\partial V_i}{\partial g_k},\delta g_k\Big\rangle dt. 
    \end{split}
\end{equation}

Using the variations of $u_i$, (i.e., $\delta u_i=\dot{\eta}_i+\ad_{u_i}\eta_i$), applying integration by parts and by the definition of the co-adjoint action the first two terms yield:

\[\int_0^T\Big\langle-\frac{d}{dt}\Big(\frac{\partial C_i}{\partial u_i}+\lambda_i\Big) + \ad^*_{u_i}\Big(\frac{\partial C_i}{\partial u_i}+\lambda_i\Big),\eta_i\Big\rangle dt.\]

From the variations $\delta\alpha_i=\rho'^*_{\eta_i}(\alpha_i)$ the third term gives

\begin{equation*}
\Big\langle\frac{\partial\Vezero}{\partial\alpha_i},\delta\alpha_i\Big\rangle=\Big\langle\frac{\partial\Vezero}{\partial\alpha_i},\rho'^*_{\eta_i}(\alpha_i)\Big\rangle=\Big\langle\alpha_i,\rho'_{\eta_i}\Big(\frac{\partial\Vezero}{\partial\alpha_i}\Big)\Big\rangle\\
=\Big\langle\textbf{J}_V\Big(\frac{\partial\Vezero}{\partial\alpha_i},\alpha_i\Big),\eta_i\Big\rangle.
\end{equation*}

Taking into account that $T(L_{g_i}\circ L_{g_i^{-1}})=TL_{g_i}\circ TL_{g_i^{-1}}$ is equivalent to the identity map on $TG_i$ and $\eta_i=T_{g_i}L_{g_i^{-1}}(\delta g_i)$, the forth term can be written as

\begin{equation*}
    \begin{split}
        \sum_{k=2}^{s}\Big\langle\frac{\partial V_{i}}{\partial g_k},\delta g_k\Big\rangle&=\sum_{k=2}^{s}\Big\langle\frac{\partial V_{i}}{\partial g_k},(T_{\overline{e}_k}L_{g_k}\circ T_{g_k}L_{g_k^{-1}})(\delta g_k)\Big\rangle=\sum_{k=2}^{s}\Big\langle\frac{\partial V_{i}}{\partial g_k},T_{\overline{e}_k}L_{g_k}(\eta_k)\Big\rangle\\
        &=\sum_{k=2}^{s}\Big\langle T^*_{\overline{e}_k}L_{g_k}\bigg(\frac{\partial V_{i}}{\partial g_k}\bigg),\eta_k\Big\rangle.\\
    \end{split}
\end{equation*}

Therefore, after performing a change of variables between indexes $i$ and $k$ in the fourth term, the above variational principle (\ref{red-var-pr=0}) yields
\begin{equation*}
    \begin{split}
        \frac{d}{dt}\Big(\frac{\partial C_i}{\partial u_i}+\lambda_i\Big) =& \ad^*_{u_i}\Big(\frac{\partial C_i}{\partial u_i}+\lambda_i\Big) +\textbf{J}_V\Big(\frac{\partial\Vezero}{\partial\alpha_i},\alpha_i\Big).
    \end{split}
\end{equation*} 
for $i=1$. Otherwise,
\begin{equation*}
        \frac{d}{dt}\Big(\frac{\partial C_i}{\partial u_i}+\lambda_i\Big) = \ad^*_{u_i}\Big(\frac{\partial C_i}{\partial u_i}+\lambda_i\Big) +\textbf{J}_V\Big(\frac{\partial\Vezero}{\partial\alpha_i},\alpha_i\Big)+\sum_{k=1}^{s}T^*_{\overline{e}_i}L_{g_i}\bigg(\frac{\partial V_{k}}{\partial g_i}\bigg).
\end{equation*} 

Finally, by taking the time derivative of $\alpha_i=\rho_{g_i}(\alpha_0)$, we have $\dot{\alpha}_i = \rho_{u_i}^{\prime*}(\alpha_i)$, together with $\alpha(0)=\rho^{*}_{g_{0}^{i}}(\alpha_{0}^{i}).$
\end{proof}

Note that the above Euler-Poincar\'{e} equations (\ref{E-P eqs}) cannot, in fact, describe the motion properly because there are more unknowns than equations. In particular, observe that equations \eqref{E-P eqs} together with \eqref{kin-each-agent} (or equivalently \eqref{kin-each-agent-basis}), give rise to only two equations for the three unknown variables $u_i$, $\lambda_i$ and $g_i$. However, we provide an additional structure to the Lie algebra, $\g$, that allows one to decouple equations \eqref{E-P eqs} into two equation. The next Proposition describes this process.

\begin{proposition}\label{split prop}
	If the Lie algebra admits a decomposition $\g=\mathfrak{r}\oplus\mathfrak{s}$ where $\mathfrak{r}=\spn\{e_1,\dots,e_m\}$ and $\mathfrak{s}=\spn\{e_{m+1},\dots,e_s\}$ such that \begin{equation}\label{algebra decomp}
		[\mathfrak{s},\mathfrak{s}]\subseteq\mathfrak{s},\hspace{0.8cm} [\mathfrak{s},\mathfrak{r}]\subseteq\mathfrak{r}, \hspace{0.8cm} [\mathfrak{r},\mathfrak{r}]\subseteq\mathfrak{s},
	\end{equation}	\noindent then the Euler-Poincar\'{e} equations of motion (\ref{E-P eqs}) are given by the following equations: 	\begin{equation}
	\begin{split}
	\frac{d}{dt}\frac{\partial C_i}{\partial u_i}=&\ad^*_{u_i}\lambda_i\bigg|_{\mathfrak{r}^*}+\mathbf{J}_V\Big(\frac{\partial\Vezero}{\partial\alpha_i},\alpha_i\Big)\bigg|_{\mathfrak{r}^*}+\Theta^{i}_{1}\sum_{k=1}^{s}T^*_{\overline{e}_i}L_{g_i}\bigg(\frac{\partial V_{k}}{\partial g_i}\bigg)\bigg|_{\mathfrak{r}^*},\\ 
 \dot{\lambda_i}=&\ad^*_{u_i}\frac{\partial C_i}{\partial u_i}\bigg|_{\mathfrak{s}^*}+\mathbf{J}_V\Big(\frac{\partial\Vezero}{\partial\alpha_i},\alpha_i\Big)\bigg|_{\mathfrak{s}^*}+\Theta^{i}_{1}\sum_{k=1}^{s}T^*_{\overline{e}_i}L_{g_i}\bigg(\frac{\partial V_{k}}{\partial g_i}\bigg)\bigg|_{\mathfrak{s}^*},\label{eqq-1}\\
	\end{split}
	\end{equation}	where $\big|_{\mathfrak{r}^*}$ and $\big|_{\mathfrak{s}^*}$ means that the expression of the corresponding factors in the last equations are written in terms of the dual space of the generators of the subspaces $\mathfrak{r}$ and $\mathfrak{s}$, respectively.
\end{proposition}

\begin{remark}
Note that semisimple Lie algebras admit a Cartan decomposition, i.e., if $\mathfrak{g}$ is semisimple, then $\mathfrak{g} = \mathfrak{r}\oplus\mathfrak{s}$ such that $
[\mathfrak{r},\mathfrak{r}]\subseteq\mathfrak{s}, \hspace{5pt} [\mathfrak{s},\mathfrak{r}]\subseteq\mathfrak{r}, \hspace{5pt} [\mathfrak{s},\mathfrak{s}]\subseteq\mathfrak{s},\nonumber$ where $\mathfrak{r} = \{x\in\mathfrak{g}\mid\theta(x) = -x\}$ is the $-1$ eigenspace of the Cartan involution $\theta$ and $\mathfrak{s} = \{x\in\mathfrak{g}\mid\theta(x) = x\}$ is the $+1$ eigenspace of the Cartan involution $\theta$. In addition, the Killing form is positive definite on $\mathfrak{r}$ and negative definite on $\mathfrak{s}$ (see, e.g., \cite{B}). So, connected semisimple Lie groups are potential candidates that satisfy the assumption of Proposition \ref{split prop}. Conversely, a Cartan decomposition determines a Cartan involution $\theta$ (see, e.g., \cite{Knapp2002}). In particular the proposed  decomposition for the Lie algebra is not restrictive in the sense that the usual manifolds/work-spaces used in applications as $SO(n)$ and $SE(n)$ allow such a decomposition. \hfill$\diamond$
\end{remark}
\begin{proof}
Given $\g=\mathfrak{r}\oplus\mathfrak{s}$ we get $\g^*=\mathfrak{r}^*\oplus\mathfrak{s}^*$, where  $\mathfrak{r}^*=\spn\{e^1,\dots,e^m\}$ and $\mathfrak{s}^*=\spn\{e^{m+1},\dots,e^s\}$. Thus, from (\ref{algebra decomp}) we have that $\ad^*_{\mathfrak{s}}\mathfrak{s}^*\subseteq\mathfrak{s}^*, \; \ad^*_{\mathfrak{s}}\mathfrak{r}^*\subseteq\mathfrak{r}^*, \; \ad^*_{\mathfrak{r}}\mathfrak{s}^*\subseteq\mathfrak{r}^*, \; \ad^*_{\mathfrak{r}}\mathfrak{r}^*\subseteq\mathfrak{s}^*$ and given that $\frac{\partial C_i}{\partial u_i}\in\mathfrak{r}^*$ and $\lambda_i\in\mathfrak{s}^*$ by definition we conclude that $\ad^*_{u_i}\frac{\partial C_i}{\partial u_i}\in\mathfrak{s}^*$ and $\ad^*_{u_i}\lambda_i\in\mathfrak{r}^*.$ Also, $\mathbf{J}_V\Big(\frac{\partial\Vezero}{\partial\alpha_i},\alpha_i\Big) \in\g^*$ and $\sum_{k=1}^{s}T^*_{\overline{e}_i}L_{g_i}\bigg(\frac{\partial V_{k}}{\partial g_i}\bigg)\in\g^*$ hence, they have a decomposition into $\mathfrak{r}^*$ and $\mathfrak{s}^*$. Thus, the equations (\ref{E-P eqs}) split into the following equations
	
	\begin{equation*}
	\begin{split}
	\frac{d}{dt}\frac{\partial C_i}{\partial u_i}=&\ad^*_{u_i}\lambda_i\bigg|_{\mathfrak{r}^*}+\mathbf{J}_V\Big(\frac{\partial\Vezero}{\partial\alpha_i},\alpha_i\Big)\bigg|_{\mathfrak{r}^*}+\sum_{k=1}^{s}T^*_{\overline{e}_i}L_{g_i}\bigg(\frac{\partial V_{k}}{\partial g_i}\bigg)\bigg|_{\mathfrak{r}^*}, \\
	\dot{\lambda_i}=&\ad^*_{u_i}\frac{\partial C_i}{\partial u_i}\bigg|_{\mathfrak{s}^*}+\mathbf{J}_V\Big(\frac{\partial\Vezero}{\partial\alpha_i},\alpha_i\Big)\bigg|_{\mathfrak{s}^*}+\sum_{k=1}^{s}T^*_{\overline{e}_i}L_{g_i}\bigg(\frac{\partial V_{k}}{\partial g_i}\bigg)\bigg|_{\mathfrak{s}^*}.
	\end{split}
	\end{equation*}
 \end{proof}

\begin{remark}
For the initial value problem guaranteeing a solution for the previous system of equations, we must solve the equations  with the initial condition $u_i(0)=T_{g(0)}L_{g^{-1}(0)}(\dot{g}_i(0))$ and the kinematic equation  $\dot{g}_i(t)=T_{\overline{e}_i}L_{g_i(t)}(u_i(t))$  with $g(0)=(g_1(0),\dots,g_s(0))$, which is a differential equation with time-dependent coefficients.  \hfill$\diamond$
\end{remark} 

\section{Discrete-time reduced necessary conditions}\label{sec5}

In this section we study the discrete-time reduction by symmetries for necessary conditions in the collision and obstacle avoidance optimal control problem. The goal is to construct a variational integrator based on the discretization of the augmented cost functional. Such integrator inherits discrete-time symmetries from its continuous counterpart and generates a well-defined (local) flow for reduced necessary conditions characterizing (local) extrema in the optimal control problem.

\subsection{Trajectory discretization}\label{sec6.1}
Given the set $\mathcal{T}=\{t_k\in\mathbb{R}^{+},\, t_{k}=kh\mid k=0,\ldots,N\}$, $Nh=T$, with $T$ fixed (recall that $T\in\mathbb{R}^{+}$ is the end point of the cost functional - see for instance equation \eqref{OCP}), a  discrete trajectory for the agent $i$ is determined by a set of $N+1$ points equally spaced in time, $g^{0:N}_i=\{g^0_i,\ldots,g^{N}_i\}$, where $g^k_i\simeq g_i(kh)\in G$, and $h=T/N$ is the time step. The path between two adjacent points $g_i^k$ and $g_i^{k+1}$ must be given by a curve lying on the Lie group $G$. To construct such a curve we make use of a retraction map $\mathcal{R}:\mathfrak{g}\to G$.
\begin{definition}\label{retractionmap}
A \textit{retraction map} $\mathcal{R}:\mathfrak{g}\to G$ is an 
analytic local diffeomorphism assigning a neighborhood $\mathcal{O}\subset\mathfrak{g}$
of $0\in {\mathfrak g}$ to a neighborhood of the identity $\overline{e}\in G$.
\end{definition}

\begin{figure}[htb!] 
\begin{center}
\begin{tikzpicture}[scale=0.8]
\draw[very thick] (4,0) -- (7,0)--(8.5,2)--(5.5,2)--(4,0);
\draw [orange]  (5.75,0.9) circle (17pt);
\filldraw [blue]  (5.75,0.9) circle (1pt) node[below]{$0$};
\filldraw [red] (5.5,1.1) circle(1pt);
\filldraw [red](5.5,1.2)node[right]{$hu^k$};
\draw (6.5,0.55) node[right]{$\mathcal{O}$};
\draw (5.5,0) node[below]{$\mathfrak{g}$};
\draw[very thick] (9,0) .. controls (9.2,0.6) and (10,1.8) .. (11,2);
\draw[very thick] (9,0) .. controls (9.3,0.3) and (10.7,0.3) .. (11,0);
\draw[very thick] (11,0) .. controls (11.5,0.6) and (12.7,1.8) .. (13,2);
\draw[very thick] (11,2) .. controls (11.8,2.2) and (13.3,2.2) .. (13,2);
\draw [orange]  (10.75,1) circle (20pt);
\filldraw [blue]  (10.75,0.9) circle (1pt) node[below]{$\overline{e}$};
\filldraw [red] (10.4,1.1) circle(1pt); 
\filldraw [red](10.3,1.2)node[right]{$\xi^{k,k+1}$};
\draw (11.2,1.55) node[right]{$\mathcal{R}\left(\mathcal{O}\right)$};
\draw (10.5,0) node[below]{$G$};
\draw [->]  (5.75,0.9) .. controls (7,-0.1) and (7,-0.1) .. (10.75,0.9);
\draw [<-](5.5,1.1) .. controls (7,3) and (7,3) .. (10.45,1.1);
\draw (7,3) node [below]{$\mathcal{R}^{-1}$};
\draw (8,0.2) node [below]{$\mathcal{R}$};
\end{tikzpicture}\caption{Retraction map.}\label{retractionfigure}
\end{center}
\end{figure}
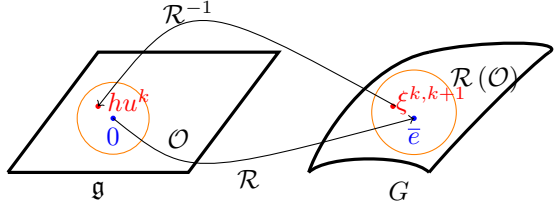

The retraction map  (see Figure \ref{retractionfigure}) is used to express small discrete changes in the
group configuration through unique Lie algebra elements given by $u^k=\mathcal{R}^{-1}((g^{k})^{-1}g^{k+1})/h$, where
$u^k\in\mathfrak{g}$ (see \cite{Rabee, KobM} for
further details). That is, if $u^{k}$ were regarded as an
average velocity between $g^{k}$ and $g^{k+1}$, then $\mathcal{R}$ is an
approximation to the integral flow of the dynamics. The difference
$\xi^{k,k+1}:=(g^{k})^{-1}\,g^{k+1}\in G$, which is an element of a nonlinear
space, can now be represented by a vector space element $u^{k}$. For the derivation of the
discrete equations of motion, the {\it right trivialized}
tangent retraction map will be used. It is the  function
$d\mathcal{R}:\mathfrak{g}\times\mathfrak{g}\rightarrow\mathfrak{g}$ given by \begin{equation}\label{righttrivialization}
T\mathcal{R}(\xi)\cdot\eta=TR_{\mathcal{R}(\xi)}d\mathcal{R}_{\xi}(\eta),
\end{equation}
where $\eta\in\mathfrak{g}$ and $R:G\to G$ the right translation on $G$ (see \cite{Rabee, KobM} for the derivation of such a map). Here we use the following notation,
$d\mathcal{R}_\xi:=d\mathcal{R}(\xi):\mathfrak{g}\rightarrow\mathfrak{g}.$ The
function $d\mathcal{R}$ is linear, but only on one argument. 

\begin{remark}
The natural choice of a retraction map is the exponential map at
the identity $\overline{e}$ of the group $G,$ $\e_{\overline{e}}:\mathfrak{g}\rightarrow
G$. Recall that, for a finite-dimensional Lie group, $\e_{\overline{e}}$ is
locally a diffeomorphism and gives rise to a natural chart
\cite{M-EP}. Then, there exists a neighborhood $U$ of $\overline{e}\in G$ such
that $\e_{\overline{e}}^{-1}:U\rightarrow \e_{\overline{e}}^{-1}(U)$ is a local
$\mathcal{C}^{\infty}-$diffeomorphism. A chart at $g\in G$ is given
by $\Psi_{g}=\e_{\overline{e}}^{-1}\circ L_{g^{-1}}.$

In general, it is not easy to work with the exponential map since the differential of the exponential map involves power series expansions with iterated Lie-brackets. In
consequence it will be useful to use a different retraction map.
More concretely, the Cayley map, which is usually used in numerical integration with matrix Lie-groups configurations (see \cite{Rabee, KobM} for further
details), will provide to us a proper framework in the application shown
in the next Section.
\end{remark}
\subsection{Discretization of the Lagrangian function}

Next, we consider a discrete cost function to construct variational integrators in
the same way as in discrete mechanics \cite{mawest}. In other words, consider the continuous-time Lagrangian $\mathbf{L}:G^s\times\mathfrak{g}^s\to\mathbb{R}$ defined by the cost functional \eqref{OCP}, that is, \begin{equation*}\mathbf{L}(g,u)=\sum_{i=1}^s\Big(C_i(g_i(t),u_i(t))+V_i^{0}(g_i)+\frac{1}{2}\sum_{j\in \mathcal{N}_i}V_{ij}(g_i(t),g_j(t))\Big),\end{equation*} and for a given $h>0$ we define the discrete Lagrangian
$L_d:G^s\times \mathfrak{g}^s\to\mathbb{R}$ as an approximation of the cost functional \eqref{OCP} 
along each discrete segment between $g^k$ and $g^{k+1}$, that is, $$L_{d}(g^{k},u^{k}) =h\mathbf{L}\left(\kappa(g^k,u^{k}),\zeta(g^k,u^{k})\right) \simeq \int_{kh}^{(k+1)h}\mathbf{L}(g,u)\,dt, $$
where $\kappa$ and $\zeta$ are functions of
$(g^k,u^{k})\in G^s\times \mathfrak{g}^s$ which approximate the
configuration $g(t)$ and the control input $u(t)$, respectively. In the following we consider a discretization given by
\begin{equation}\label{discretelagrangian}
    L_d(g^k,u^{k})=h\sum_{i=1}^{s}\left(C_{i}(g_{i}^k,u_{i}^k) +V_{i}^{0}(g_{i}^{k}) + V_{i}(g^{k})\right).
\end{equation}


\subsection{Discrete-time optimal control problem and reduction of discrete-time necessary conditions for optimality}

Next, we are going to define the optimal control problem for discrete-time systems and derive a variational integrator for $L_{d}\colon G^s\times \mathfrak{g}^s\to\R$, in a similar fashion as the variational principle presented in Theorem \ref{main-theo}.

\textbf{Problem:} Consider the discrete-time optimal control problem for collision and obstacle avoidance of left-invariant multi agent control systems which is given by finding the discrete configurations $\{g^k\}_{k=0}^{N}=\{(g_1^k,\ldots,g_s^k)\}_{k=0}^{N}$ and discrete control inputs $\{u^k\}_{k=0}^{N}=\{(u_1^k,\ldots,u_s^k)\}_{k=0}^{N}$ minimizing the discrete cost functional
\begin{equation}\label{ocpdiscrete}
    \min_{(g^k, u^k)}\sum_{i=1}^{s}\sum_{k=0}^{N-1}h\left(C_i(g_{i}^k ,u_{i}^k) +     V_{i}^{0}(g_{i}^{k}) + V_{i}(g^{k}) \right)
\end{equation}
subject to $g_i^{k+1} =g_i^k\mathcal{R}(hu_i^k)$ (i.e., a first order approximation of equation \eqref{kin-each-agent}) with given boundary conditions
$g^0$ and $g^N$, where $h>0$ denotes the time step, $\mathcal{R}:\mathfrak{g}\to G$ is a retraction map, $u_i(0)$ and $u_i(T)$ are given, and each cost function $C_{i}:G\times\mathfrak{g}\to\mathbb{R}$, potential functions $V_{i}^{0}$ and $V_{i}$ satisfy properties (i) - (vii).  \hfill$\square$

The discrete-time optimal control problem \eqref{ocpdiscrete} can be considered as a discrete constrained variational problem by introducing the Lagrange multipliers $\mu_{i}^k\in \mathfrak{g}$ into the cost functional. Consider the augmented discrete Lagrangian $\mathcal{L}_{d}:G^s\times G^s\times \mathfrak{g}^s\times (\mathfrak{g}^{*})^s\to\mathbb{R}$ given by
\begin{equation}
    \begin{split}
    \mathcal{L}_d (g^k,g^{k+1},u^k,&\mu^k)=  h\sum_{i=1}^{s}\left(C_{i}(g_{i}^k,u_{i}^k) +V_{i}^{0}(g_{i}^{k}) \right. \\
     &\quad +\left. V_{i}(g^{k}) +\Big{\langle}\mu^k_i,\frac{1}{h}\mathcal{R}^{-1}(\xi^{k,k+1}_i)-u_{i}^k\Big{\rangle} \right)
\end{split}
\end{equation}
where $\xi^{k,k+1}_i=(g_i^k)^{-1}g_i^{k+1}\in G$, for each $i\in\mathcal{N}$. Note that the last term in the augmented Lagrangian represents a first-order discretization of the kinematic constraint paired with a Lagrange multiplier in analogy with the variational principle presented in Section \ref{sec4}.

Now, extending the potential $V_{i}^{0}$ we obtain an extended Lagrangian $\mathcal{L}_{ext,d}:G^s\times G^s\times\mathfrak{g}^s\times(\mathfrak{g}^{*})^s\times (V^{*})^{s}\to\mathbb{R}$ given by 
\begin{equation}
    \begin{split}
    \mathcal{L}_{ext,d} & (g^k,g^{k+1},u^k,\mu^k, \alpha)=  h\sum_{i=1}^{s}\left(C_{i}(g_{i}^k,u_{i}^k) \right. \\ + & V_{i}^{0,ext}(g_{i}^{k},\alpha_i) 
     +  \left. V_{i}(g^{k}) +\Big{\langle}\mu^k_i,\frac{1}{h}\mathcal{R}^{-1}(\xi^{k,k+1}_i)-u_{i}^k\Big{\rangle} \right),
\end{split}
\end{equation}
which is invariant under the left action of $G$ on $G^s\times G^s\times \mathfrak{g}^s \times (\mathfrak{g}^{*})^s \times (V^{*})^s$ given by 
$\tilde{\Phi}_{g}(h_{1},h_{2},u,\mu,\alpha)=(gh_{1},gh_{2},u,\mu,\rho_{g^{-1}}^{*}(\alpha))$ by assumption (vi). In particular, under assumptions (iv)-(vi), the extended discrete Lagrangian $\mathcal{L}_{ext,d}(\cdot,\cdot,\cdot,\cdot,\alpha_0)=:\mathcal{L}_{ext,d,\alpha_0}$ is $G_{\alpha_0}$-invariant under $\tilde{\Phi}_g$.

The following result (Theorem \ref{theorem2}) derives a variational integration for reduced optimality conditions for the discrete-time optimal control \eqref{ocpdiscrete} in analogy with the results presented in Section \ref{sec4}. To derive the numerical algorithm, first we need the following result describing variations for elements on the Lie algebra and its relation with variations on the Lie group by using the retraction map, in addition to a property used in the proof of Theorem \ref{theorem2}.

\begin{lemma}[adapted from \cite{Rabee, KobM}] \label{lemmadiscrete}The following properties hold 
\begin{enumerate}
\item[(i)] $$\frac{1}{h}\delta\left(\mathcal{R}^{-1}(\xi^{k,k+1})\right)=\frac{1}{h}d\mathcal{R}^{-1}_{(hu^k)}(-\eta^k+\hbox{Ad}_{\mathcal{R}(hu^k)}\eta^{k+1}),$$ where $\eta^{k}=T_{g^k}L_{(g^k)^{-1}}(\delta g^k)\in\mathfrak{g}^s$ and $\xi^{k,k+1}=(g^k)^{-1}g^{k+1}$.
\item[(ii)] $$(d\mathcal{R}^{-1}_{(-hu^k)})^{*}\mu^k=\Ad^{*}_{\mathcal{R}(hu^k)}(d\mathcal{R}^{-1}_{(hu^k)})^{*}\mu^{k}$$ where $\mu^k\in(\mathfrak{g}^{*})^s$ and $d\mathcal{R}^{-1}$ is the inverse right trivialized  tangent of the retraction map $\mathcal{R}$ defined in \eqref{righttrivialization}. 
\end{enumerate}
\end{lemma}

\begin{theorem}\label{theorem2}
Under assumptions (i)-(vii), a normal extrema for the discrete-time optimal control problem \eqref{ocpdiscrete} satisfies the following equations
\begin{align}
g_i^{k+1} =&g_i^k\mathcal{R}(hu_i^k),\label{reconstructiondiscrete}\\
\left(d\mathcal{R}^{-1}_{(hu_i^k)}\right)^{*}\mu_i^k=&\left(d\mathcal{R}^{-1}_{(-hu_i^{k-1})}\right)^{*}\mu_i^{k-1} +
\mathbf{J}_V\left(h\frac{\partial V^{0,ext}_{i}}{\partial \bar\alpha_i^k},  \bar\alpha_i^k\right)\nonumber \\ & + h\Theta_{1}^{i}\sum_{l=1}^{s} T_{\overline{e}_{i}}^{*}L_{g_i^k}\left(\frac{\partial V_{l}}{\partial g_i^k}\right),\label{eqlpdiscrete}\\
\mu_i^k=&\left(\frac{\partial C_i}{\partial u^k_i}\right),\label{momentumdiscrete}\\
\bar\alpha_i^{k+1}=&\rho^{*}_{\mathcal{R}(hu_i^k)}(\bar\alpha_i^k),\;\; \bar\alpha_i^0=\rho_{g_i^0}^{*}(\alpha_i^0),\label{alphadiscrete}
\end{align} for $k=1,\ldots,N-1$; where $\Theta^{i}_{1}=0$ if $i=1$,  otherwise $\Theta^{i}_{1}=1$.
\end{theorem}


\begin{proof}
Since the cost functions and the potential functions satisfy assumptions (i) - (vii), as in the continuous-time case, it is possible to induce the reduced augmented discrete Lagrangian $\ell_{ext,d}:G^{s-1}\times G^s \times\mathfrak{g}^s\times(\mathfrak{g}^{*})^s\times (V^{*})^s\to\mathbb{R}$ as 
\begin{align*}
    \ell_{ext,d}(g^{k},\xi^{k,k+1}, u^k,\mu^k,\bar\alpha^{k})=& h\sum_{i=1}^{s}(C_{i}(u_{i}^k) + V_{i}^{0,ext}(\bar\alpha^{k}_i) \\&+\Big\langle \mu_{i}^k,\frac{1}{h}\mathcal{R}^{-1}(\xi_{i}^{k,k+1}) -u_{i}^k\Big\rangle + V_{i}(g^{k}),\end{align*}
where $\bar\alpha^{k}_i=\rho_{g^{k}_i}^{*}(\alpha_i^0)$ for a fixed $\alpha_i^0\in V^{*}$ satisfying $\alpha_i^0=\rho^*_{g_i^0}(\alpha_i^0)$ and, with a slight abuse of notation, $C_{i}(u_{i}^k) = C_{i}(e_{i}, u_{i}^k)$ and $V_{i}^{0,ext}(\alpha^{k}_i)=V_{i}^{0,ext}(e, \alpha^{k}_i)$. Notice that, also here, $g_{1}^{k}$ is set to be the identity element, so that we have $g^{k}\in G^{s-1}$.

As in the proof for Theorem \ref{main-theo}, the technical part is to show that a normal extrema of the reduced variational principle
\begin{equation}\label{discrete:reduced:action}
    \delta \sum_{k=0}^{N-1} \ell_{ext,d}(g^{k}, \xi^{k,k+1}, u^k,\mu^k,\bar\alpha^{k}) = 0
\end{equation}
satisfies equations \eqref{reconstructiondiscrete}-\eqref{momentumdiscrete} for all variations of $\mathcal{R}^{-1}(\xi^{k,k+1})$ (induced by variations of $g^k$ vanishing at the endpoints), $u^k$ and $\bar{\alpha}^k$ of the form $\rho_{\eta^{k}}^{'*}(\bar\alpha^{k})$ where $\eta^k\in\mathfrak{g}^s$ vanishes at the endpoints. Then, similarly as in the proof for Theorem \ref{main-theo}, it follows that a normal extrema for the optimal control problem \eqref{ocpdiscrete} satisfies the variational principle
$$\delta \sum_{k=0}^{N-1} \mathcal{L}_{d} (g^k,g^{k+1},u^k,\mu^k) = 0,$$ for all variations of $g^k$ (vanishing at the endpoints), $\mathcal{R}^{-1}(\xi^{k,k+1})$ (induced by variations of $g^k$) and $u^k$.

Note that 
\begin{equation*}
    \begin{split}
        0  & = \delta \sum_{k=0}^{N-1} \ell_{ext,d}(g^{k}, \xi^{k,k+1}, u^k,\mu^k,\bar\alpha^{k}) \\
        & = \sum_{k=0}^{N-1}\sum_{i=1}^{s}h\left[ \Big\langle \frac{\partial C_{i}}{\partial u_{i}^{k}} - \mu_{i}^{k}, \delta u_{i}^{k} \Big\rangle + \Big\langle \frac{\partial V_{i}^{0,ext}}{\partial \bar\alpha^{k}}, \delta \bar\alpha^{k}\Big\rangle + \sum_{l=2}^{s}\Big \langle \frac{\partial V_{i}}{\partial g_{l}^{k}}, \delta g_{l}^{k} \Big\rangle \right.\\
        & + \left. \Big\langle \mu_{i}^{k}, \frac{1}{h} d\mathcal{R}_{h u_{i}^{k}}^{-1}(-\eta_{i}^{k}+\Ad_{\mathcal{R}(h u_{i}^{k})}\eta_{i}^{k+1})\Big\rangle \right]
    \end{split}
\end{equation*}
where we used Lemma \ref{lemmadiscrete} to obtain the last term. Since variations $\delta u_{i}^{k}$ are arbitrary we obtain $\mu_{i}^{k} = \frac{\partial C_{i}}{\partial u_{i}^{k}}$. As for the second term we have that 
\begin{equation*}\Big\langle \frac{\partial V_{i}^{0,ext}}{\partial \bar\alpha^{k}}, \delta \bar\alpha^{k}\Big\rangle = \Big\langle \frac{\partial V_{i}^{0,ext}}{\partial \bar\alpha^{k}}, \rho_{\eta_{i}^{k}}^{'*}(\bar\alpha^{k})\Big\rangle = \Big\langle \mathbf{J}_V\left(\frac{\partial V_{i}^{0,ext}}{\partial \bar\alpha^{k}},\bar\alpha^{k}\right),\eta_{i}^{k} \Big\rangle.\end{equation*}
As we have show along the proof for Theorem \ref{main-theo}, we have that
\begin{equation*}
        \sum_{l=2}^{s} \Big\langle \frac{\partial V_{i}}{\partial g_{l}^{k}}, \delta g_{l}^{k} \Big\rangle= \sum_{l=2}^{s} \Big\langle T^*_{\overline{e}_l}L_{g_l^{k}}\bigg(\frac{\partial V_{i}}{\partial g_l^{k}}\bigg),\eta_l^{k}\Big\rangle.
\end{equation*}
and the indexes might be interchanged. The last term to obtain equation \eqref{eqlpdiscrete} may be dealt with, using integration by parts in discrete-time, which is just rearranging the indexes, together with the second statement in Lemma \ref{lemmadiscrete}, and  therefore, it follows the derivation of equations \eqref{reconstructiondiscrete}-\eqref{momentumdiscrete}.
\end{proof}
\begin{remark}
Equations \eqref{reconstructiondiscrete}-\eqref{momentumdiscrete} are as a discrete approximation of the Lie-Poisson equations for the Hamiltonian version of the optimal control problem considered in \cite{SCD}. The equation $\displaystyle{\mu_i^k=\left(\frac{\partial C_i}{\partial u^k_i}\right)}$ represents the discrete time version of the reduced Legendre transformation and the equation $g^{k+1}_i=g^k_i\mathcal{R}(hu^k_i)$ is the analogous of the reconstruction equation in the discrete time counterpart. These three equations are used to compute $u_i^{k},\mu_i^{k}$ and $g^{k+1}_i$ given $u_i^{k-1}$, $\mu_i^{k-1}$, $g^{k-1}_i$ and $g^{k}_i$ from $k=1$ to $k=N-1$.\hfill$\diamond$ \end{remark}

To compute the discrete-time reduced necessary condition for the optimal control problem \eqref{ocpdiscrete} we must enforce boundary conditions given by the continuous-time quantities. More precisely, we must set

\begin{equation}
    \begin{split}
        \left(d\mathcal{R}^{-1}_{(hu_i^{0})}\right)^{*}\mu_i^0=&\frac{\partial C_i}{\partial u_i}(u_i(0))+h\Theta_{1}^{i} \sum_{l=2}^{s} T_{\overline{e}_{i}}^{*}L_{g_i^0}\left(\frac{\partial V_{l}}{\partial g_i^0}\right) + h \mathbf{J}_V\left(\frac{\partial V_{i}^{0,ext}}{\partial \bar\alpha^{0}},\bar\alpha^{0}\right),\\
        \frac{\partial C_i}{\partial u_i}(u_i(T))=&\left(d\mathcal{R}^{-1}_{(-hu_i^{N-1})}\right)^{*}\mu^{N-1}_i\label{vc2},
    \end{split}
\end{equation}
relating  the  momenta  at  the initial and  final  times, and  used  to transform boundary values between the continuous and discrete representation. They follow from the principle that any variation with free boundary points of the action \eqref{discrete:reduced:action} along a solution of equations \eqref{reconstructiondiscrete}-\eqref{alphadiscrete} equals the change in momentum $\langle \frac{\partial C_i}{\partial u_i}(u_i(T)), g_{i}^N \delta g_{i}^{N} \rangle - \langle \frac{\partial C_i}{\partial u_i}(u_i(0)), g_{i}^N \delta g_{i}^{N} \rangle$ (see \cite{KobM} for a discussion in the single agent case).

\begin{remark}
If we choose the midpoint rule to discretize the potential $V_{i}$, then we would obtain the following boundary conditions
\begin{equation*}
    \begin{split}
        \left(d\mathcal{R}^{-1}_{(hu_i^{0})}\right)^{*}\mu_i^0=&\frac{\partial C_i}{\partial u_i}(u_i(0))+\frac{h}{2}\Theta_{1}^{i} \sum_{l=1}^{s} T_{\overline{e}_{i}}^{*}L_{g_i^0}\left(\frac{\partial V_{l}}{\partial g_i^0}\right) + h \mathbf{J}_V\left(\frac{\partial V_{i}^{0,ext}}{\partial \bar\alpha^{0}},\bar\alpha^{0}\right),\\
        \frac{\partial C_i}{\partial u_i}(u_i(T))=&\left(d\mathcal{R}^{-1}_{(-hu_i^{N-1})}\right)^{*}\mu^{N-1}_i  + \frac{h}{2}\Theta_{1}^{i} \sum_{l=1}^{s} T_{\overline{e}_{i}}^{*}L_{g_i^N}\left(\frac{\partial V_{l}}{\partial g_i^N}\right).\hfill\diamond
    \end{split}
\end{equation*}
\end{remark}

The boundary condition $g_s(T)$ for agent $s$ is enforced by the relation \begin{equation}\label{gT}\mathcal{R}^{-1}((g^N_s)^{-1}g_s(T))=0.\end{equation} Recalling that $\mathcal{R}(0)=\overline{e}_s$, this last expression just means that $g^N_s=g_s(T)$. Moreover, by computing recursively the equation $g^{k+1}_s=g^{k}_s\mathcal{R}(hu_s^k)$ for $k=1,\ldots,N-1$, using that $g_s^{0}=g_s(0)$ and $\mathcal{R}(0)=\overline{e}_s$, it is possible to translate the final configuration $g_s^N$ in terms of $u^k_s$ such that there is no need to optimize over any of the configurations $g^k_s$. In that sense, \eqref{eqlpdiscrete} for $i=s$, $l=0$ together with \begin{equation}\label{gNeq}
\mathcal{R}^{-1}\left[\mathcal{R}(h(u^{N-1}_s)^{-1}\ldots\mathcal{R}(hu^0_s)^{-1}(g_s(0))^{-1}g_s(T)\right]=0,
\end{equation} form a set of $(nN)$-equations (since $\dim\mathfrak{g}=n$) where $nN$ unknowns are for $u^{0:N-1}_s$. 


The numerical algorithm to compute the reduced optimality conditions is summarized in Algorithm \ref{Algorithm}.

\vspace{-0.2cm}%
\begin{algorithm}
\small
\caption{Reduced conditions for optimality}
\label{Algorithm}
\begin{algorithmic}[1]
\State \textbf{Data:} Lie group $G$, its Lie algebra $\mathfrak{g}$, cost functions $C_i$, artificial potential functions $V_{ij}$, $V_{i,\hbox{ext}}^{0}$, final time $T$, $\#$ of steps $N$.\\ \textbf{inputs:} $g_i(0)$, $g_i(T)$, $u_i(0)$, $u_i(T)$, $\alpha^0_i$, $u_{i}^{0}$ for all $i=1,\ldots,s$ and $h=T/N$.
    \For {$i=1 \to s$}
        \State Fix $g_{i}^{0}=g_{i}(0)$ and $\alpha^0_i$
        \State solve \eqref{reconstructiondiscrete} and \eqref{alphadiscrete} for $k=0$.
    \EndFor
    \State \textbf{outputs:} $g^1_s$, $\alpha^1_s$
    \For {$k=1 \to N-1$}
        \For {$i= 1 \to s$} 
                    \State solve \eqref{reconstructiondiscrete}-\eqref{alphadiscrete} subject to \eqref{vc2}.
        \EndFor
    \EndFor
    \State \textbf{outputs:} $g_s^{0:N-1}$, $u_s^{0:N-1}$, $\mu_s^{0:N-1}, \alpha_s^{0:N-1}$.
    \State \textbf{Compute} $g_s^{0:N-1}$, $u_s^{0:N-1}$, $\mu_s^{0:N-1}, \alpha_s^{0:N-1}$ \textbf{subjected} to \eqref{gNeq}.
\Statex
\end{algorithmic}
\end{algorithm}

Note also that the exact form of equations \eqref{reconstructiondiscrete}-\eqref{momentumdiscrete} depends on the choice of $\mathcal{R}$. This choice will also influence the computational efficiency of the optimization framework when the above equalities are enforced as constraints. For instance, in Section \ref{sec6}, we will employ the Cayley transform on the Lie group $SE(2)$ as a choice of $\mathcal{R}$ to write in a compact form the numerical integrator \cite{Rabee}, \cite{KM}, but another natural choice would be to employ the exponential map, as we explained in Section \ref{sec6.1}.


\section{Case Study}\label{sec6}

In this case study we apply the proposed reduction by symmetry strategy to an optimal control for autonomous surface vehicles (ASVs). The configuration space whose elements determine the motion of each ASV is $SE(2)\cong SO(2)\times\R$. An element $g_i\in SE(2)$ is given by $g_i=\begin{pmatrix}
\cos\theta_i&-\sin\theta_i&x_i\\
\sin\theta_i&\cos\theta_i&y_i\\
0          &          0& 1
\end{pmatrix}$, where $(x_i,y_i)\in\R^2$ represents the center of mass of a planar rigid body describing the ASV and $\theta_i$ represents the angular orientation of the ASV. The control inputs, for each ASV, are given by $u_i=(u_i^1,u_i^2)$ where $u_i^1$ denotes the speed of the center of mass for the ASV and $u_i^2$ denotes the angular velocity of the ASV. 

The kinematic equations for the multi-agent system are:
\begin{equation}\label{kinem-eqs}
  \dot{x_i}=u_i^2\cos\theta_i, \; \; \dot{y_i}=u_i^2\sin\theta_i,  \; \; \dot{\theta_i}=u_i^1,\, i=1,\ldots,s.  
\end{equation}

Using the notation of Example \ref{example1}, the Lie algebra $\se$ is identified with $\R^2$ through the isomorphism $\begin{pmatrix}
-a\mathbb{J}&b\\
0&0
\end{pmatrix}\mapsto (a,b).$
The elements of the basis of the Lie algebra $\se$ are $\displaystyle{e_1=\begin{pmatrix}
0&-1&0\\
1&0&0\\
0&0&0
\end{pmatrix}, 
e_2=\begin{pmatrix}
0&0&1\\
0&0&0\\
0&0&0
\end{pmatrix},
e_3=\begin{pmatrix}
0&0&0\\
0&0&1\\
0&0&0
\end{pmatrix}}$,

\noindent which satisfy $[e_1,e_2]=e_3,\; [e_2,e_3]=0, \; [e_3,e_1]=e_2$. Thus, the kinematic equations (\ref{kinem-eqs}) take the form $\dot{g_i}=g_iu_i=g_i(u_i^1e_1+u_i^2e_2)$ and give rise to a left-invariant control system on $SE(2)^s\times \mathfrak{se}(2)^s$. The inner product on $\se$ is given by $\langle\langle\xi_1,\xi_2\rangle\rangle:=tr(\xi_1^T\xi_2)$ for $\xi_1,\xi_2\in\se$ and hence, the norm is given by $||\xi||=\sqrt{tr(\xi^T\xi)},$ for any $\xi\in\se$. The dual Lie algebra $\se^*$ of $SE(2)$ is defined through the dual pairing, $\langle\alpha,\xi\rangle=tr(\alpha^{T}\xi)$, where $\alpha\in\se^*$ and $\xi\in\se$ hence, the elements of the basis of $\se^*$ are 
$\displaystyle{e^1=\begin{pmatrix}
0&\frac{1}{2}&0\\
-\frac{1}{2}&0&0\\
0&0&0
\end{pmatrix}, 
e^2=\begin{pmatrix}
0&0&0\\
0&0&0\\
1&0&0
\end{pmatrix},
e^3=\begin{pmatrix}
0&0&0\\
0&0&0\\
0&1&0
\end{pmatrix}}$.

Consider the cost function $C_i(g_i,u_i)=\frac{1}{2}\langle u_i,u_i\rangle$ and the artificial potential function $V_{ij}:SE(2)\times SE(2)\to\R$ given by $\displaystyle{V_{ij}(g_i,g_j)=\frac{\sigma_{ij}}{2((x_i-x_j)^2+(y_i-y_j)^2-4\overline{r}^2)}}$, where $\sigma_{ij}\in\R_{\geq 0}$ and $\overline{r}$ is the radius of the disk each agent occupies as defined at the end of Section III. Consider a spherical obstacle with unit radius and without loss of generality let it be centered at the origin. Hence, consider the obstacle avoidance potential function $V_i^0:SE(2)\to\R$, $\displaystyle{V_i^0(g_i)=\frac{\sigma_{i0}}{2(x^2+y^2-(\overline{r}+1)^2)}}$, where $\sigma_{i0}\in\mathbb{R}_{>0}$. 

Note that the obstacle avoidance potential functions are not $SE(2)$-invariant but $SO(2)$-invariant, so they break the symmetry. Using the norm of $\se$ and for $\overline{r}=1$, $V_{ij}$ and $V_i^0$ are equivalently given by
$\displaystyle{V_{ij}(g_i,g_j)=\frac{\sigma_{ij}}{2(||\Ad_{g_i^{-1}g_j}e_1||^2-6)}}$ and 
$\displaystyle{V_i^0(g_i)=\frac{\sigma_{i0}}{2(||\Ad_{g^{-1}_i}e_1||^2-6)}}$.

Let $V=\se^*$, so we define the extended potential functions $\Vezero:SE(2)\times \se\to\R$ by
$\Vezero(g_i,\alpha)=\frac{\sigma_{i0}}{2(||\Ad_{g^{-1}_i}\alpha||^2-6)}$, which are $SE(2)$-invariant under the action of $\tilde{\Phi}$ given by \eqref{eq_phi_tilde}, i.e. $\Vezero\circ\tilde{\Phi}=\Vezero,$ for any $g\in SE(2)$. Since, $V=\se^*$ we have $\displaystyle{\mathbf{J}_V\Big(\frac{\partial\Vezero}{\partial\alpha_i},\alpha_i\Big)=\ad^*_{\alpha_i}\Big(\frac{\partial\Vezero}{\partial\alpha_i}\Big)}$, and equations (\ref{E-P eqs}) and (\ref{E-P eqs param}) yield \begin{align*}\frac{d}{dt}\Big(u_i+\lambda_i\Big) =& \ad^*_{u_i}\Big(u_i+\lambda_i\Big) +\ad^*_{\alpha_i}\Big(\frac{\partial\Vezero}{\partial\alpha_i}\Big)\\&+\Theta_1^i\sum_{j\in\mathcal{N}_i}T^*_{\overline{e}_i}L_{g_i}\bigg(\frac{\partial V_{ij}}{\partial \theta_i}e^1+\frac{\partial V_{ij}}{\partial x_i}e^2+\frac{\partial V_{ij}}{\partial y_i}e^3\bigg),\end{align*}together with $\dot{\alpha}_i=-\ad_{u_i}\alpha_i$ and $\alpha_i=\Ad_{g_i^{-1}}\alpha_0$. 

Note also that $u_i=u_i^1e_1+u_i^2e_2$, $\alpha_i=\alpha_i^1e_1+\alpha_i^2e_2+\alpha_i^3e_3$ and $\lambda_i=\lambda_3^ie^3$ thus, \begin{align*}
    ad^*_{u_i}(u_i+\lambda_i) &= \begin{pmatrix}
0&-\frac{u_i^2\lambda_3^i}{2}&0\\
\frac{u_i^2\lambda_3^i}{2}&0&0\\
u_i^1\lambda_3^i&-u_i^1u_i^2&0
\end{pmatrix}, \,\, \ad_{u_i}\alpha_i = \begin{pmatrix}
0&0&-u_i^1\alpha_i^3\\
0&0&u_i^1\alpha_i^2-u_i^2\alpha_i^1\\
0&0&0
\end{pmatrix}, \\
 & \\
\ad^*_{\alpha_i}\Big(\frac{\partial\Vezero}{\partial\alpha_i}\Big) &= \begin{pmatrix}
		0&0&0\\
		0&0&0\\
		\Gamma_{i,0}^{31}&\Gamma_{i,0}^{32}&0
		\end{pmatrix} =\frac{\sigma_{i0}\alpha_i^1}{(||\alpha_i||^2-6)^2}\begin{pmatrix}
		0&0&0\\
		0&0&0\\
		-\alpha_i^3&\alpha_i^2&0
		\end{pmatrix}, \\
 & \\
T^*_{\overline{e}_i}L_{g_i}\Big(\frac{\partial V_{ij}}{\partial g_i}\Big) &= \begin{pmatrix}
0&0&0\\
0&0&0\\
\Gamma_{ij}^{31}&\Gamma_{ij}^{32}&0
\end{pmatrix} = \frac{-\sigma_{ij}}{((x_{ij})^2+(y_{ij})^2-4)^2}\begin{pmatrix}
0&0&0\\
0&0&0\\
x_{ij}&y_{ij}&0
\end{pmatrix},
\end{align*}where $x_{ij}=x_i-x_j, \; y_{ij}=y_i-y_j$ and $\displaystyle{\Ad_{g_i^{-1}}\alpha_0=\begin{pmatrix}
0&-1&x_i\sin\theta_i-y_i\cos\theta_i\\
1&0&x_i\cos\theta_i+y_i\sin\theta_i\\
0&0&0
\end{pmatrix}}$.

Therefore, by applying Proposition \ref{split prop} for $\mathfrak{r}=\{e_1,e_2\}$ and $\mathfrak{s}=\{e_3\}$, Euler-Lagrange equations for the OCP (\ref{OCP}) are

\begin{equation*}
\dot{u}_i^1=-\frac{u_i^2\lambda_3^i}{2}, 
\quad\dot{u}_i^2=u_i^1\lambda_3^i+\Theta_1^i\Gamma_{i,0}^{31}+\sum_{j\in\mathcal{N}_i}\Gamma_{ij}^{31}, \,\,\dot{\lambda}_3^i=-u_i^1u_i^2+\Theta_1^i\Gamma_{i,0}^{32}+\sum_{j\in\mathcal{N}_i}\Gamma_{ij}^{32}, 
\end{equation*}

\noindent with 
\vspace{0.3cm}

$\begin{array}{ll}
	\dot{\alpha}_i^1=0, & \alpha_i^1(0)=1, \\
	\dot{\alpha}_i^2=u_i^1\alpha_i^3,  &  \alpha_i^2(0)=x_i^0\sin\theta_i^0-y_i^0\cos\theta_i^0,  \\
	\dot{\alpha}_i^3=-u_i^1\alpha_i^2+u_i^2\alpha_i^1, & \alpha_i^3(0)=x_i^0\cos\theta_i^0+y_i^0\sin\theta_i^0. 
\end{array}$

\vspace{0.3cm}

For the discrete-time setting, one would choose
\begin{equation*}
C_{i}(g^{k}_i,u^{k}_i) = \frac{h}{2}\langle u^{k}_{ij},u^{k}_{ij}\rangle,\
V_{i}(g^{k}_i) = \frac{h\sigma_i}{2(\|\Ad_{(g^{k}_i)^{-1}}e_{1}\|^{2}-6)},\nonumber
\end{equation*}
where $\displaystyle{g^{k}_i = \begin{bmatrix}
        \cos\theta^{k}_i & -\sin\theta^{k}_i & \phantom{-}x^{k}_i\\
        \sin\theta^{k}_i & \phantom{-}\cos\theta^{k}_i & \phantom{-}y^{k}_i\\
        0 & \phantom{-}0 & \phantom{-}1
        \end{bmatrix}\in SE(2)\nonumber}$ and $\displaystyle{u^{k}_i = \sum_{j=1}^{3}u^{k}_{ij}e_{j}\in\mathfrak{se}(2)}$. Also, in the discrete-time setting, the extended potential function $V_{d,\textnormal{ext}}: SE(2)\times \mathfrak{se}(2)\to\mathbb{R}$ can be constructed in exactly the same way as in the above example and is given by
\begin{align}
V_{i}^{0,ext}(g_{i}^{k},\alpha_i) = \frac{h\sigma_i}{2(\|\Ad_{(g^{k}_i)^{-1}}\alpha_i\|^{2}-6)},\nonumber
\end{align}
where $
\displaystyle{\alpha_{i}^{k} = \sum_{j=1}^{3} \alpha_{ij}^{k}e^{j}}$. We do not give all the details again and leave it up to reader to verify that the assumptions  (i) - (vii) from \eqref{a1} are satisfied. The discrete-time equations are 
\begin{align}
(d\mathcal{R}_{hu^{k}_i}^{-1})^{*}\mu^{k}_i &= (d\mathcal{R}_{-hu^{k-1}_i}^{-1})^{*}\mu^{k-1}_i+\ad_{\bar{\alpha}^{k}_i}^{*}\frac{\partial V_{i}^{0,ext}}{\partial\bar{\alpha}^{k}_i}\nonumber\\
 &= \Theta_1^i\sum_{j\in\mathcal{N}_i}T^*_{\overline{e}_i}L_{g_i}\bigg(\frac{\partial V_{ij}}{\partial \theta_i}e^1+\frac{\partial V_{ij}}{\partial x_i}e^2+\frac{\partial V_{ij}}{\partial y_i}e^3\bigg)\nonumber,\\
\bar{\alpha}^{k+1}_i &= \Ad_{\mathcal{R}(hu^{k}_i)^{-1}}\bar{\alpha}^{k}_i, \hspace{5pt} \bar{\alpha}^{0}_i = \Ad_{(g^{0}_i)^{-1}}\alpha^{0}_i\nonumber,
\end{align}
where
\begin{align}
\ad_{\bar{\alpha}^{k}_i}^{*}\frac{\partial V_{i}^{0,ext}}{\partial\bar{\alpha}^{k}_i} 
= \frac{h\sigma_i\bar{\alpha}^{k}_{i1}}{(\|\bar{\alpha}^k_i\|^{2}-6)^{2}}\begin{bmatrix}
                                                                    \phantom{-}0 & \phantom{-}0  & \phantom{-}0\\
                                                                    \phantom{-}0  & \phantom{-}0 & \phantom{-}0\\
                                                                    -\bar{\alpha}^{k}_{i3} & \phantom{-}\bar{\alpha}^{k}_{i2} & \phantom{-}0
                                                                    \end{bmatrix}.\nonumber
\end{align}

For numerical purposes, one first chooses an appropriate retraction map, such as the Cayley map or the exponential map, and then computes the quantities $(d\mathcal{R}_{hu^{k}_i}^{-1})^{*}\mu^{k}_i$ and $(d\mathcal{R}_{-hu^{k-1}_i}^{-1})^{*}\mu^{k-1}_i$. As an example, if we choose the Cayley map $\text{cay}: \mathfrak{se}(2)\to\mathrm{SE}(2)$ as the retraction map, (see \cite{Rabee} and \cite{KM} for instance) then we have


\begin{align}
[d\text{cay}_{hu_{i}^{k}}^{-1}]^{*}\mu_{i}^{k} = \begin{bmatrix}
            0 & \frac{1}{2}\gamma_i & 0\\ -\frac{1}{2}\gamma_i & \phantom{-}0 & \phantom{-}0\\
            \mu_{i2}^{k}-\frac{h u_{i1}^{k} \mu_{i3}^{k}}{2} & \frac{h u_{i1}^{k} \mu_{i2}^{k}}{2}+\mu_{i3}^{k} & \phantom{-}0 \end{bmatrix}\nonumber
\end{align}
where \begin{equation*}\gamma_i = \left(\frac{h^2 (u_{i1}^{k})^{2}}{4}+1\right)\mu_{i1}^{k}+\left(\frac{h^2 u_{i1}^{k} u_{i2}^{k}}{4}-\frac{h u_{i3}^{k}}{2}\right)\mu_{i2}^{k}+\left(\frac{h^2 u_{i1}^{k} u_{i3}^{k}}{4}+\frac{h u_{i2}^{k}}{2}\right)\mu_{i3}^{k}\end{equation*} and $
\displaystyle{\mu_{i}^{k} = \sum_{j=1}^{3} \mu_{ij}^{k}e^{j}}$.
Note that for $\displaystyle{v = \sum_{i=1}^{3}v^{i}e_{i}\in\mathfrak{g}}$, the matrix representation for $d\text{cay}_{v}^{-1}$ is given by
\begin{align}
[d\text{cay}_{v}^{-1}] = \begin{bmatrix}
                         1+\dfrac{(v^1)^2}{4} & \phantom{-}0  & \phantom{-}0\\
                         \dfrac{v^1v^2}{4}-\dfrac{v^3}{2} & \phantom{-}1 & \phantom{-}\dfrac{v^1}{2}\\
                         \dfrac{v^1v^3}{4}+\dfrac{v^2}{2} & -\dfrac{v^1}{2} & \phantom{-}1
                         \end{bmatrix}.\nonumber
\end{align}

\section{Conclusions}\label{sec7}

We studied the reduction by symmetry for optimality conditions of extrema in an OCP for collision and obstacle avoidance of left-invariant multi-agent control system on Lie groups, by exploiting the physical symmetries of the agents and obstacles. Reduced optimality conditions are obtained using techniques from variational calculus and Lagrangian mechanics on Lie groups, in the continuous-time and discrete-time settings. We applied the results to an OCP for multiple unmanned surface vehicles. The method proposed in this work allows the construction of position and velocity estimators, by discretizing the variational principle given in Theorem \ref{main-theo} - instead of discretizing the equations of motion - and by deriving variational integrators 
 - see Theorem \ref{theorem2}. The reduction of sufficient conditions for optimality will be also studied by using the notion of  conjugate points as in \cite{borum} in future work, as well as the reduction by symmetry of the variational obstacle avoidance problems \cite{jacob} on semidirect products of Lie groups endowed with a bi-invariant metric on a Riemannian manifold.

\section*{Acknowledgments}
L. Colombo is very grateful to A. Bloch, R. Gupta and T. Ohsawa for many useful comments and stimulating discussions during the last years on the topic of this paper, which is inspired by our common previous work \cite{BCGO}.


\begin{thebibliography}{00}


\bibitem{B}
{\sc A. M. Bloch}, {\em Nonholonomic mechanics and control,} Springer-Verlag New York, 2015.

\bibitem{BCGO}
{\sc A. M. Bloch, L. J. Colombo, R. Gupta, T.Ohsawa}, {\em Optimal control problems with symmetry breaking cost functions,} SIAM J. Applied Algebra and Geometry, 1 (2017), 626-646.

\bibitem{BlControlled} {\sc A. M. Bloch, D. E. Chang, N. E. Leonard, and J. E. Marsden}, {\em Controlled Lagrangians and the
stabilization of mechanical systems. II. Potential shaping,} IEEE Transactions on Automatic
Control, 46(10):1556–1571, Oct 2001.

\bibitem{bonnabel1}{\sc S. Bonnabel, P.M. Silvere, and P. Rouchon,} {\em Symmetry-preserving observers.} IEEE Transactions on Automatic Control, 53(11), 2514-2526, 2008.

\bibitem{borum}
{\sc A. Borum, T. Bretl}, {\em Reduction of sufficient conditions for optimal control problems with subgroup symmetry}. IEEE Transactions on Automatic Control, 62 (2017), 3209--3224.




\bibitem{Rabee} {\sc N. Bou-Rabee and Marsden, J.E.} {\em Hamilton–Pontryagin integrators on Lie groups part I: Introduction and structure-preserving properties. Foundations of computational mathematics}, vol 9(2), pp.197-219, 2009.

\bibitem{leo1}{\sc L. Colombo and D. V. Dimarogonas}, {\em Symmetry Reduction in Optimal Control of Multiagent Systems on Lie Groups,} in IEEE Transactions on Automatic Control, vol. $65$, no. $11$, pp. $4973-4980$, 2020.
\bibitem{contreras}{\sc C. Contreras, T. Ohsawa.} {\em Controlled Lagrangians and stabilization of Euler–Poincaré mechanical systems with broken symmetry II: potential shaping}. Mathematics of Control, Signals, and Systems, pp.1-31, 2022.
\bibitem{contreras2} {\sc C. Contreras and T. Ohsawa}.{\em Stabilization of Mechanical Systems on Semidirect Product Lie Groups with Broken Symmetry via Controlled Lagrangians}. IFAC-PapersOnLine, 54(19), 106-112, 2021.

\bibitem{eche}{\sc A. Echeverr\'ia-Enr\'iquez, J. Mar\'in-Solano, M. C. Munoz-Lecanda, and N. Rom\'an-Roy,} {\em Geometric Reduction in optimal control theory with symmetries.} Reports on Mathematical Physics, 52(1), 89-113, 2003.

\bibitem{arcak} {\sc R. Ferreira, A., Meissen, C., Arcak, M., Packard, A.} {\em Symmetry reduction for performance certification of interconnected systems}. IEEE Transactions on Control of Networked Systems, 5 (2017), 525--535.

\bibitem{GB} {\sc F. Gay-Balmaz and T. S. Ratiu,} {\em Clebsch optimal control formulation in mechanics,} Journal of Geometric Mechanics, 3(1), 41–79, 2011.

\bibitem{GBT}{\sc F. Gay-Balmaz and Cesare Tronci,} {\em Reduction theory for symmetry breaking with applications to nematic systems}. Physica D: Nonlinear Phenomena, 239(20):1929-1947, 2010.

\bibitem{grizzle}{\sc J. W. Grizzle and S. I. Marcus.}  {\em The structure of nonlinear control systems possessing symmetries.} IEEE Trans. Auto. Control, 30(3):248--258, 1985.
\bibitem{jacob} {\sc J. Goodman and L. Colombo.} {\em Collision Avoidance of Multiagent Systems on Riemannian Manifolds}. SIAM Journal on Control and Optimization 60(1), 168-188, 2022.
\bibitem{H} 
{\sc D. Holm}, {\em Geometric Mechanics, Part II}, Imperial College Press, 2008.

\bibitem{HMR}
{\sc D. Holm, J. E. Marsden, T. S. Ratiu}, {\em The Euler-Poincar\'{e} equations and semidirect products with application to continuum theories}, Adv. Math., 137 (1998), pp. 1 - 81.
\bibitem{HSS}
{\sc D. Holm, T. Schmah, C. Stoica}, {\em Geometric mechanics and symmetry,} Oxford University Press, 2009.
\bibitem{jurdjevic} {\sc V.~Jurdjevic}, {\em Geometric control theory}, Cambridge University, 1997.

\bibitem{LG1} {\sc A. Khosravian, J. Trumpf, R. Mahony, and T. Hamel,} {\em State estimation for invariant systems on lie groups with delayed output measurements}. Automatica, 68:254–265, 2016.
\bibitem{Knapp2002}  {\sc A. W. Knapp,} {\em Lie Groups Beyond an Introduction}, Birkhauser Boston, Boston, 2002.

\bibitem{KobM}{\sc M. Kobilarov, J. Marsden}. {\em Discrete geometric optimal control on Lie groups}. IEEE Transactions on Robotics 27.4 (2011): 641-655.
\bibitem{KM}
{\sc W.-S. Koon and J. E. Marsden,} {\em Optimal control for holonomic and nonholonomic mechanical systems
with symmetry and Lagrangian reduction,} SIAM Journal on Control and Optimization, 35,901–929, 1997.
\bibitem{JK2}
{\sc E.~Justh, P.~Krishnaprasad}, {\em Optimality, reduction and collective
  motion}, Proc. R. Soc. A, 471 (2015), 20140606.

\bibitem{K}
{\sc P. S. Krishnaprasad,} {\em Optimal Control and Poisson Reduction,} Technical Report T.R. 93-87, Institute
for Systems Research, University of Maryland, College Park, MD, 1993.

\bibitem{LG2} {\sc C. Lageman, J. Trumpf, and R. Mahony,}  {\em Gradient-like observers for invariant dynamics on a Lie group}. IEEE Trans. on Aut. Contr., 55(2):367-377, 2010.

\bibitem{manolo}{\sc M. de Le\'on, J. Cort\'es, D. Mart\'in de Diego, and S. Mart\'inez,} {\em General symmetries in optimal control,} Reports on Mathematical Physics, vol. 53, no. 1, pp. 55–78, 2004.

\bibitem{Leonard1}
{\sc N.~Leonard, P.~Krishnaprasad}, {\em Motion control of drift-free,
  left-invariant systems on lie groups}, IEEE Transactions on Automatic
  Control, 40 (1995), 1539--1554.

\bibitem{LG3} {\sc R. Mahony, T. Hamel, and J.-M. Pflimlin,}  {\em Non-linear complementary filters on the special orthogonal group.} IEEE Transactions on Automatic Control, 53(5):1203–1218, 2008.
\bibitem{M-EP} {\sc J. Marsden, S. Pekarsky, and S. Shkoller}, {\em Discrete Euler-Poincar\'e and lie-poisson equations}. Nonlinearity, 12(6), 1647, 1999.
\bibitem{MR}
{\sc J. Marsden, T. Ratiu,} {\em Introduction to Mechanics and Symmetry,} Springer-Verlag, 1999.

\bibitem{MRW1}
{\sc J. Marsden, T. Ratiu, A. Weinstein,} {\em Reduction and Hamiltonian structures on duals of
semidirect product Lie algebras,} in Fluids and Plasmas: Geometry and Dynamics, Contemp. Math. 28,
American Mathematical Society, Providence, RI, 28 (1984), pp. 55–100.

\bibitem{MRW2}
{\sc J. Marsden, T. Ratiu, A. Weinstein,} {\em Semidirect products and reduction in
mechanics,} Trans. Amer. Math. Soc., 281, 147–177, 1984. 
\bibitem{mawest}
{\sc J. Marsden and M. West}.{\em  Discrete Mechanics and variational
integrators.}  Acta Numerica, Vol.10, pp. 357--514, (2001).
\bibitem{Tomoki}
{\sc T.~Ohsawa}, {\em Symmetry reduction of optimal control systems and
  principal connections}, SIAM J. Control and Optimization, 51, (2012), 96-120.
\bibitem{Tomoki-CDC} {\sc T.~Ohsawa}, {\em Poisson Reduction of Optimal Control Systems}, 50th IEEE Conference on Decision and Control and European Control Conference, pp. 6230–6235, 2011.

\bibitem{O-SM}
{\sc R. Olfati-Saber, R. M. Murray,} {\em Distributed cooperative control of multiple vehicle formations using structural potential functions}, IFAC world congress, 15 (2002), 242-248.

\bibitem{saccon}{\sc A Saccon, J Hauser, AP Aguiar}, Optimal control on Lie groups: The projection operator approach,  {\em IEEE Transactions on Automatic Control}, 58 (2013), 2230-2245.
  
  \bibitem{SCD}{\sc E. Stratoglou,  L. Colombo, T. Ohsawa}, Optimal Control with Broken Symmetry of Multi-Agent Systems on Lie Groups. {\em arXiv preprint arXiv:2204.06050}, 2022.
  

\bibitem{bonnabel} {\sc A. Sarlette, S. Bonnabel, and R. Sepulchre,} {\em Coordinated motion design on lie groups.} IEEE Trans. on Automatic Control, 55(5):1047–1058, 2010.

\bibitem{arian} {\sc  A. J. van der Schaft,} {\em Symmetries and conservation laws for hamiltonian systems with inputs and outputs: A generalization of Noether’s theorem}. Sys. Contr. Lett., 1:108–115, 1981.


\bibitem{vasile1}{\sc C. Vasile, M. Schwager, C. Belta}. {\em SE(N) invariance in networked systems}. In 2015 European Control Conference, 186-191, 2015.
\bibitem{vasile2} {\sc C. Vasile, M. Schwager, C. Belta}. {\em Translational and rotational invariance in networked dynamical systems}. IEEE Transactions on Control of Network Systems, 5(3), 822-832, 2017.




\end{thebibliography}
\end{document}